\documentclass[a4,12pt]{amsart}
\usepackage{amssymb,amstext,amsmath,amscd,amsthm,amsfonts,enumerate,latexsym}
\usepackage{color}
\usepackage[dvipdfmx]{graphicx}
\usepackage[all]{xy}
\theoremstyle{plain}
\newtheorem{thm}{Theorem}[section]

\newtheorem*{thm*}{Theorem}
\newtheorem*{cor*}{Corollary}

\newtheorem{prop}[thm]{Proposition}
\newtheorem{proposition}[thm]{Proposition}
\newtheorem{lemma}[thm]{Lemma}

\newtheorem{corollary}[thm]{Corollary}
\newtheorem{claim}{Claim}
\newtheorem*{claim*}{Claim}

\theoremstyle{definition}

\newtheorem{definition}[thm]{Definition}
\newtheorem{ex}[thm]{Example}
\newtheorem{Example}[thm]{Example}

\newtheorem{remark}[thm]{Remark}

\newtheorem{setting}[thm]{Setting}

\theoremstyle{remark}

\numberwithin{equation}{thm}

\def\Ext{\operatorname{Ext}}

\def\Ker{\operatorname{Ker}}

\def\Hom{\operatorname{Hom}}

\def\Max{\operatorname{Max}}

\def\mod{\mathrm{mod}}

\def\rank{\mathrm{rank}}

\def\c{\mathfrak c}

\def\m{\mathfrak m}
\def\n{\mathfrak n}

\newcommand{\rma}{\mathrm{a}}

\newcommand{\rmc}{\mathrm{c}}

\newcommand{\rme}{\mathrm{e}}
\newcommand{\rmf}{\mathrm{f}}

\newcommand{\rmr}{\mathrm{r}}

\newcommand{\rmv}{\mathrm{v}}

\newcommand{\rmI}{\mathrm{I}}
\newcommand{\rmJ}{\mathrm{J}}
\newcommand{\rmK}{\mathrm{K}}

\newcommand{\rmQ}{\mathrm{Q}}

\newcommand{\calB}{\mathcal{B}}

\newcommand{\calR}{\mathcal{R}}
\newcommand{\calS}{\mathcal{S}}
\newcommand{\calT}{\mathcal{T}}

\newcommand{\fka}{\mathfrak{a}}

\newcommand{\fkc}{\mathfrak{c}}

\newcommand{\fkm}{\mathfrak{m}}
\newcommand{\fkn}{\mathfrak{n}}

\newcommand{\fkp}{\mathfrak{p}}
\newcommand{\fkq}{\mathfrak{q}}

\newcommand{\mapright}[1]{%
\smash{\mathop{%
\hbox to 1cm{\rightarrowfill}}\limits^{#1}}}

\newcommand{\mapleft}[1]{%
\smash{\mathop{%
\hbox to 1cm{\leftarrowfill}}\limits_{#1}}}

\def\AGL{\operatorname{AGL}}

\def\Ass{\operatorname{Ass}}

\title{Sally modules of canonical ideals in dimension one and $2$-AGL rings}
\author{Tran Do Minh Chau}
\address{Thai Nguyen University of Education, Khoa Toan, truong DHSP Thai Nguyen}
\email{trandominhchau@gmail.com}
\author{Shiro Goto}
\address{Department of Mathematics, School of Science and Technology, Meiji University, 1-1-1 Higashi-mita, Tama-ku, Kawasaki 214-8571, Japan}
\email{shirogoto@gmail.com}
\author{Shinya Kumashiro}
\address{Department of Mathematics and Informatics, Graduate School of Science and Technology, Chiba University, Chiba-shi 263, Japan}
\email{polar1412@gmail.com}
\author{Naoyuki Matsuoka}
\address{Department of Mathematics, School of Science and Technology, Meiji University, 1-1-1 Higashi-mita, Tama-ku, Kawasaki 214-8571, Japan}
\email{naomatsu@meiji.ac.jp}
\thanks{The first author was partially supported by the International Research Supporting Program of Meiji University. The second author was partially supported by JSPS Grant-in-Aid for Scientific Research (C) 25400051. The fourth author was partially supported by JSPS Grant-in-Aid for Scientific Research 26400054. }

\begin{document}
\maketitle

\setlength{\baselineskip} {17pt}

\begin{abstract}
The notion of $2$-$\AGL$ ring in dimension one which is a natural generalization of almost Gorenstein local ring is posed in terms of the rank of Sally modules of canonical ideals. The basic theory is developed, investigating also the case where the rings considered are numerical semigroup rings over fields. Examples are explored. 
\end{abstract}

{\footnotesize \tableofcontents}

\section{Introduction}\label{section1}
The destination of this research is to find a good notion of Cohen-Macaulay local rings of positive dimension which naturally generalizes Gorenstein local rings. In dimension one, the research has started from the works of V. Barucci and R. Fr\"{o}berg \cite{BF} and the second, the fourth authors, and T. T. Phuong \cite{GMP}. In \cite{BF} Barucci and Fr\"{o}berg introduced the notion of almost Gorenstein ring in the case where the local rings are one-dimensional and analytically unramified. They explored also numerical semigroup rings over fields and developed a beautiful theory. In \cite{GMP} the authors extended the notion given by \cite{BF} to arbitrary one-dimensional Cohen-Macaulay local rings and showed that their new definition works well to analyze almost Gorenstein rings which are analytically ramified. In \cite{GTT} the second author, R. Takahashi, and N. Taniguchi gave the notion of almost Gorenstein local/graded rings of higher dimension. Their research is still in progress, exploring, for example, the problem of when the Rees algebras of ideals/modules are almost Gorenstein rings; see \cite{GMTY1, GMTY2, GMTY3, GMTY4, GRTT, T}. One can consult \cite{El} for a deep investigation of canonical ideals in dimension one.

The interests of the present research are a little different from theirs and has been strongly inspired by \cite[Section 4]{GGHV} and \cite{V2}. Our aim is to discover a good candidate for natural generalization of almost Gorenstein rings. Even though our results are at this moment restricted within the case of dimension one, we expect that a higher dimensional notion might be possible after suitable modifications. However, before entering more precise discussions, let us fix our terminology.

Throughout this paper let $(R,\m)$ be a Cohen-Macaulay local ring of dimension one and let $I$ be a canonical ideal of $R$. Assume that $I$ contains a parameter ideal $Q = (a)$ of $R$ as a reduction. We set $K = \frac{I}{a} = \{\frac{x}{a} \mid x \in I\}$ in the total ring $\rmQ(R)$ of fractions of $R$ and let $S =R[K]$. Therefore, $K$ is a fractional ideal of $R$ such that $R \subseteq K \subseteq \overline{R}$ and $S$ is a module-finite extension of $R$, where $\overline{R}$ denotes the integral closure of $R$ in $\rmQ(R)$. We denote by $\fkc = R:S$ the conductor. With this notation the second and the fourth authors and T. T. Phuong  \cite{GMP} closely studied the almost Gorenstein property of $R$. Here let us recall the definition of almost Gorenstein local rings given by \cite{GTT}, which works with a suitable modification in higher dimensional cases also. Notice that in our setting, the condition in Definition \ref{1.1} below is equivalent to saying that $\fkm K \subseteq R$ (\cite[Proposition 3.4]{GTT}).

\begin{definition}[{\cite[Definition 1.1]{GTT}}]\label{1.1}
Suppose that $R$ possesses the canonical module $\rmK_R$. Then we say that $R$ is an almost Gorenstein local ({\it AGL} for short) ring, if there is an exact sequence
$$0 \to R \to \rmK_R \to C \to 0$$
of $R$-modules such that $\fkm C = (0)$.
\end{definition} 

\noindent
Consequently, $R$ is an AGL ring if $R$ is a Gorenstein ring (take $C=(0)$) and Definition \ref{1.1} certifies that once $R$ is an AGL ring, although it is not a Gorenstein ring, $R$ can be embedded into its canonical module $\rmK_R$ and the difference $\rmK_R/R$ is little.

Let $\rme_i(I)$~($i = 0,1$) denote the Hilbert coefficients of $R$ with respect to $I$ (notice that our canonical ideal $I$ is an $\fkm$-primary ideal of $R$) and let $\rmr(R) = \ell_R(\Ext_R^1(R/\fkm,R))$ denote the Cohen-Macaulay type of $R$. With this notation the following characterization of AGL rings is given by \cite{GMP}, which was a starting point of  the present research.

\begin{thm}[{\cite[Theorem 3.16]{GMP}}]\label{1.2} The following conditions are equivalent.
\begin{enumerate}[{\rm (1)}]
\item $R$ is an $\AGL$ ring but not a Gorenstein ring.
\item $\rme_1(I) = \rmr(R)$.
\item $\rme_1(I) = \rme_0(I) - \ell_R(R/I) + 1$.
\item $\ell_R(S/K) = 1$, that is $S = K :\fkm$.
\item $\ell_R(I^2/QI) = 1$.
\item $S = \fkm : \fkm$ but $R$ is not a $\mathrm{DVR}$. 
\end{enumerate}
When this is the case, $I^3=QI^2$ and 
$$\ell_R(R/I^{n+1}) = \left(\rmr(R) +\ell_R(R/I) -1\right)\binom{n+1}{1}-\rmr(R)$$
for all $n \ge 1$, where $\ell_R(*)$ denotes the length.
\end{thm}

The aim of the present research is to ask for a generalization of AGL rings in dimension one. For the purpose we notice that Condition (3) in Theorem \ref{1.2}  is equivalent to saying that the Sally module  of $I$ with respect to $Q$ has rank  one. In order to discuss more explicitly, here let us explain  the notion of Sally module (\cite{V1}). The results we recall below  hold true in Cohen-Macaulay local rings $(R,\m)$ of arbitrary positive dimension for all $\fkm$-primary ideals $I$ and reductions $Q$ of $I$ which are parameter ideals of $R$ (\cite{GNO}). Let us, however, restrict our attention to the case where $\dim R = 1$ and $I$ is a canonical ideal of $R$.

Let $\calT= \calR(Q)=R[Qt]$ and $\calR = \calR(I)=R[It]$ respectively denote the Rees algebras of $Q$ and $I$, where  $t$  is an indeterminate over $R$. We set $$\calS_Q(I) = I\calR/I\calT$$ and call it the Sally module of $I$ with respect to $Q$ (\cite{V1}). Then $\calS_Q(I)$ is a finitely generated graded $\calT$-module with $\dim_{\calT}\calS_Q(I) \le 1$  whose grading is given by
$$
[\calS_Q(I)]_n  =  \begin{cases}
(0)& ~if ~n \le 0, \\
I^{n+1}/Q^nI  & ~if~n \ge 1
\end{cases}
$$
for each $n \in \Bbb Z$ (\cite[Lemma 2.1]{GNO}).  
Let $\fkp = \fkm \calT$ and $\calB = \calT/\fkp$ ~($=(R/\fkm)[T]$ the polynomial ring). We set $$\rank ~~\calS_Q(I) = \ell_{\calT_\fkp}([\calS_Q(I)]_\fkp)$$ and call it the rank of $\calS_Q(I)$. Then $\Ass_{\calT}\calS_Q(I) \subseteq \{\fkp\}$ and 
$$\rank~~ \calS_Q(I)  = \rme_1(I)-\left[\rme_0(I) -\ell_R(R/I)\right]$$
(\cite[Proposition 2.2]{GNO}). As we later confirm  in Section 2 (Theorem \ref{2.3}), the invariant $\rank~\calS_Q(I)$ is equal to $\ell_R(S/K)$ and is independent of the choice of canonical ideals $I$ and their reductions $Q$. By \cite{S3, V1} it is known that Condition (3) in Theorem \ref{1.2} is equivalent to saying that $\rank~~\calS_Q(I)=1$, which is also equivalent to saying that $$\calS_Q(I) \cong \calB(-1)$$ as a graded $\calT$-module. According to these stimulating facts, as is suggested by \cite[Section 4]{GGHV} it seems reasonable to expect that one-dimensional Cohen-Macaulay local rings $R$ which satisfy the condition
$$\rank~~\calS_Q(I)= 2, \  ~\text{that~is}~\ \rme_1(I) = \rme_0(I) - \ell_R(R/I) + 2$$
for canonical ideals $I$ could be a good candidate for generalization of AGL rings.

Chasing the expectation, we now give the following.

\begin{definition}\label{1.3} We say that $R$ is a $2$-almost Gorenstein local ($2$-{\it AGL} for short) ring, if $\rank~~ \calS_Q(I) = 2$.

\end{definition}

In this paper we shall closely explore the structure of $2$-$\AGL$ rings to show  the above expectation comes true. Let us note here the basic characterization of $2$-$\AGL$ rings, which starts the present paper.

\begin{thm}\label{1.4} The following conditions are equivalent.
\begin{enumerate}[{\rm (1)}]
\item $R$ is a $2$-$\AGL$ ring. 
\item There is an exact sequence
$0 \to \calB(-1) \to \calS_Q(I) \to \calB(-1) \to 0$
of graded $\calT$-modules.
\item $K^2 = K^3$ and $\ell_R(K^2/K) = 2$.
\item $I^3=QI^2$ and $\ell_R(I^2/QI) = 2$.
\item $R$ is not a Gorenstein ring but $\ell_R(S/[K:\fkm])=1$.
\item $\ell_R(S/K)=2$.
\item $\ell_R(R/\fkc) = 2$.
\end{enumerate}
When this is the case, $\fkm{\cdot}\calS_Q(I) \ne (0)$, whence  the exact sequence given by Condition $(2)$ is not split, and we have 
$$\ell_R(R/I^{n+1}) = \rme_0(I)\binom{n+1}{1}- \left(\rme_0(I) - \ell_R(R/I) +2\right)$$
for all $n \ge 1$.
\end{thm}

See \cite{HHS} for another direction of generalization of Gorenstein rings. In \cite{HHS} the authors posed the notion of nearly Gorenstein ring and developed the theory. Here let us note that in dimension one, $2$-$\AGL$ rings are not nearly Gorenstein and nearly Gorenstein rings are not $2$-$\AGL$ rings (see \cite [Remark 6.2, Theorem 7.4] {HHS}, Theorems \ref {1.4}, \ref{2.6}).

Here let us explain how this paper is organized. In Section 2 we will summarize some preliminaries, which we need throughout this paper. The proof of Theorem \ref{1.4} shall be given in Section 3. In Section 3 we study also the question how the $2$-$\AGL$ property of rings is preserved under flat base changes. Condition (7) in Theorem \ref{1.4} is really practical, which we shall show in Sections 5 and 6. In Section 4 we study $2$-$\AGL$ rings obtained by idealization. We will show that $A=R \ltimes \fkc$ is a $2$-$\AGL$ ring if and only if so is $R$, which enables us, starting from a single $2$-$\AGL$ ring, to produce an infinite family $\{A_n\}_{n \ge 0}$ of $2$-$\AGL$ rings which are analytically ramified (Example \ref{ex4}). Let $\rmv(R)$ (resp. $\rme(R))$ denote the embedding dimension  of $R$ (resp. the multiplicity $\rme_\m^0(R)$ of $R$ with respect to $\fkm$). We set $B = \fkm :\fkm$. Then it is known by \cite[Theorem 5.1]{GMP} that $R$ is an AGL ring with $\rmv(R) = \rme(R)$ if and only if $B$ is a Gorenstein ring. In Section 5 we shall closely study the corresponding phenomenon of the $2$-$\AGL$ property. We will show that if $R$ is a $2$-$\AGL$ ring with $\rmv(R) = \rme(R)$, then $B$ contains a unique maximal ideal $M$ such that $B_N$ is a Gorenstein ring for all $N \in \Max B \setminus \{M\}$ and $B_M$ is an $\AGL$ ring which is not a Gorenstein ring. The converse is also true under suitable conditions, including the specific one that $K/R$ is a free $R/\fkc$-module. Section 6 is devoted to the analysis of the case where $R = k[[H]]$ ($k$ a field) are the semigroup rings of numerical semigroups $H$. We will give in several cases a characterization for $R = k[[H]]$ to be $2$-$\AGL$ rings in terms of numerical semigroups $H$.  
  

\section{Preliminaries}\label{section2}
The purpose of this section is to summarize some auxiliary results, which we later need throughout this paper. First of all, let us make sure of our setting.

\begin{setting}\label{2.1}
Let $(R,\m)$ be a Cohen-Macaulay local ring with $\dim R=1$, possessing the canonical module $\rmK_R$. Let $I$ be a canonical ideal of $R$. Hence $I$ is an ideal of $R$ such that $I \ne R$ and $I \cong \rmK_R$ as an $R$-module. We  assume that $I$ contains a parameter ideal $Q=(a)$ of $R$ as a reduction. Let   $$K = \frac{I}{a} = \left\{\frac{x}{a} \mid x \in I\right\}$$ in the total ring $\rmQ(R)$ of fractions of $R$. Hence $K$ is a fractional ideal of $R$ such that $R \subseteq K \subseteq \overline{R}$, where $\overline{R}$ denotes the integral closure of $R$ in $\rmQ(R)$. Let $S =R[K]$ and $\fkc = R : S$. We denote by $\calS_Q(I) = I\calR/I\calT$ the Sally module of $I$ with respect to $Q$, where $\calT = R[Qt]$, $\calR = R[It]$, and $t$ is an indeterminate over $R$. Let $\calB = \calT/\fkm \calT$ and $\rme_i(I)~ (i = 0,1) $ the Hilbert coefficients of $I$. 
\end{setting}

We notice that a one-dimensional Cohen-Macaulay local ring $(R,\fkm)$ contains a canonical ideal if and only if $\rmQ(\widehat{R})$ is a Gorenstein ring, where $\widehat{R}$ denotes the $\fkm$-adic completion of $R$ (\cite[Satz 6.21]{HK}). Also, every $\fkm$-primary ideal of $R$ contains a parameter ideal as a reduction, once the residue class field $R/\fkm$ of $R$ is infinite. If $K$ is a given  fractional ideal of $R$ such that $R \subseteq K \subseteq \overline{R}$ and $K \cong \rmK_R$ as an $R$-module, then taking a non-zerodivisor $a \in \m$ so that $a K \subsetneq R$, $I = aK$ is a canonical ideal of $R$ such that  $Q=(a)$ as a reduction and $K = \frac{I}{a}$. Therefore, the existence of canonical ideals $I$ of $R$ containing parameter ideals as reductions is equivalent to saying that there are fractional ideals $K$ of $R$ such that $R \subseteq K \subseteq \overline{R}$ and $K \cong \rmK_R$ as an $R$-module (cf. \cite[Remark 2.10]{GMP}). We have for all $n \ge 0$ $$K^{n+1}/K^n \cong I^{n+1}/QI^n$$ as an $R$-module, whence $K/R \cong I/Q$.  Let $\rmr_Q(I) = \min\{n \ge 0 \mid I^{n+1} = QI^n\}$ be the reduction number of $I$ with respect to $Q$.

Let us begin with the following.

\begin{lemma}\label{2.2} The following assertions hold true.
\begin{enumerate}[{\rm (1)}]
\item $\rmr_Q(I) = \min\{n \ge 0 \mid K^n = K^{n+1}\}$. Hence $S = K^n$ for all $n \ge \rmr_Q(I)$.
\item Let $b \in I$. Then $(b)$ is a reduction of $I$ if and only if $\frac{b}{a}$ is an invertible element of $S$. When this is the case, $S = R[\frac{I}{b}]$ and $\rmr_{Q}(I) = \rmr_{(b)}(I)$. 
\end{enumerate}
\end{lemma}

\begin{proof}
(1) The first equality is clear, since $I = aK$. The second one follows from the fact that $S = \bigcup_{n \ge 0}K^n$.

(2) Suppose that $(b)$ is a reduction of $I$ and choose an integer $n \gg 0$ so that $S = K^n$ and $I^{n+1} = bI^n$. Then since $\frac{I^{n+1}}{a^{n+1}} = \frac{b}{a}{\cdot}\frac{I^n}{a^n},$ we get 
$S = \frac{b}{a}S$, whence $\frac{b}{a}$ is  an invertible element of $S$. The reverse implication is now clear. To see $S = R[\frac{I}{b}]$, notice that $S \supseteq \frac{a}{b}{\cdot}\frac{I}{a}=\frac{I}{b}$, because $\frac{a}{b} \in S$. Hence $S \supseteq R[\frac{I}{b}]$ and by symmetry we get $S = R[\frac{I}{b}]$. To see $\rmr_{Q}(I) = \rmr_{(b)}(I)$, let $n =\rmr_Q(I)$. Then $K^{n+1} = S =\frac{b}{a}S = \frac{b}{a}K^n$ by Assertion (1) , so that $I^{n+1}=bI^n$. Therefore, $\rmr_Q(I) \ge \rmr_{(b)}(I)$, whence $\rmr_Q(I) = \rmr_{(b)}(I)$ by symmetry. 
\end{proof}

\begin{proposition}[{\cite{GMP, GTT}}]\label{2.4}
The following assertions hold true.
\begin{enumerate}[{\rm (1)}]
\item $\fkc = K : S$ and $\ell_R(R/\fkc)=\ell_R(S/K)$.
\item $\fkc = R:K$ if and only if $S = K^2$.
\item $R$ is a Gorenstein ring if and only if $\rmr_Q(I) \le 1$. When this is the case, $I = Q$, that is $K = R$. 
\item $R$ is an $\AGL$ ring if and only if $\m K^2 \subseteq K$. 
\item Suppose that $R$ is an $\AGL$ ring but not a Gorenstein ring. Then $\rmr_Q(I) = 2$ and $\ell_R(K^2/K)=1$.
\end{enumerate}
\end{proposition}

\begin{proof}
(1) See \cite[Lemma 3.5 (2)]{GMP}.

(2) Since $K : K =R$ (\cite[Bemerkung 2.5 a)]{HK}), we have $$R:K = (K:K):K = K : K^2.$$ Because $\fkc = K:S$ by Assertion (1), $\fkc =R:K$ if and only if $K:S = K:K^2$. The latter condition is equivalent to saying that $S=K^2$ (\cite[Definition 2.4]{HK}).

(3), (5) See \cite[Theorems 3.7, 3.16]{GMP}.

(4) As $K : K =R$, $\m K^2 \subseteq K$ if and only if $\m K \subseteq R$. By \cite[Proposition 3.4]{GTT} the latter condition is equivalent to saying that  $R$ is an AGL ring. 
\end{proof}

Let $\mu_R(M)$ denote, for each finitely generated $R$-module $M$, the number of elements in a minimal system of generators of $M$.

\begin{corollary}\label{2.5}
The following assertions hold true.
\begin{enumerate}[{\rm (1)}]
\item $K:\m \subseteq K^2$ if $R$ is not a Gorenstein ring. If $R$ is not an $\AGL$ ring, then $K :\fkm \subsetneq K^2$.
\item $\m K^2 + K= K :\m$, if $\ell_R(K^2/K)=2$.
\item Suppose that $R$ is not a Gorenstein ring. Then $\mu_R(S/K) = \rmr(R/\fkc)$. Therefore,  $R/\fkc$ is a Gorenstein ring if and only if $\mu_R(S/K)=1$.
\end{enumerate}

\end{corollary}

\begin{proof}
(1) As $R$ is not a Gorenstein ring, $K \ne  K^2$ by Lemma \ref{2.2} (1) and Proposition \ref{2.4} (3). Therefore,  $K : \fkm \subseteq K^2$, since  $\ell_R([K:\fkm]/K) = 1$ (\cite[Satz 3.3 c)]{HK})  and $\ell_R(K^2/K) < \infty$.  Proposition \ref{2.4} (4) implies that $K:\m \ne K^2$ if $R$ is not an $\AGL$ ring.

(2) Suppose that $\ell_R(K^2/K) = 2$. Then $R$ is not an $\AGL$ ring. Hence $\m K^2 \not\subseteq K$, while by Assertion (1) $K:\m \subsetneq K^2$. Therefore, since $\m^2 K^2 \subseteq K$, we get
$$K \subsetneq \m K^2 + K \subseteq K : \m \subsetneq K^2,$$
whence $\m K^2 + K =K : \m$, because $\ell_R(K^2/K)=2$.

(3) We get $\mu_R(S/K)= \ell_R(S/(\m S + K))=\ell_R([K : (\m S + K)]/(K:S))$, where the second equality follows by duality (\cite[Bemerkung 2.5 c)]{HK}). Since $K: K=R$ and $\fkc = K : S$ by Proposition \ref{2.4} (1), 
\begin{eqnarray*}
\mu_R(S/K) &=&\ell_R([K : (\m S + K)]/(K:S))\\
&=& \ell_R(\left[(K : \m S) \cap (K:K)\right]/\fkc)\\
&=& \ell_R(\left[(K : S):\m] \cap R\right]/\fkc)\\
&=&\rmr(R/\fkc),
\end{eqnarray*}
where the last equality follows from the fact $[(K : S):\m] \cap R = \fkc:_R\m= (0):_R\m/\fkc$.
\end{proof}

We close this section with the following, which guarantees that $\rank~\calS_Q(I)$ and $S=R[K]$ are independent of the choice of canonical ideals $I$ of $R$ and reductions $Q$ of $I$. Assertions (1) and (2) of Theorem \ref{2.3} are more or less known (see, e.g., \cite{GMP, GGHV}). In particular, in \cite{GGHV} the invariant $\ell_R(I/Q)$ is called the canonical degree of $R$ and intensively investigated. Let us include here a brief proof in our context for the sake of completeness.

\begin{thm}\label{2.3} The following assertions hold true.
\begin{enumerate}[{\rm (1)}]
\item $\ell_R(S/K) = \rme_1(I) -\left[\rme_0(I) - \ell_R(R/I)\right]$. Hence $\rank ~\calS_Q(I) = \ell_R(S/K) =\ell_R(R/\fkc)$.
\item The invariants $\rmr_Q(I)$, $\ell_R(S/K)$, and $\ell_R(K/R)$ are independent of the choice of $I$ and $Q$.
\item The ring $S=R[K]$ is independent of the choice of $I$ and $Q$.
\end{enumerate}
\end{thm}

\begin{proof}
(1) We have $K/R \cong I/Q$ as an $R$-module, whence $$\ell_R(K/R) = \ell_R(I/Q) = \ell_R(R/Q) - \ell_R(R/I)=\rme_0(I) - \ell_R(R/I).$$ So,  the first equality is clear, because $\ell_R(S/R) = \rme_1(I)$ by \cite[Lemma 2.1]{GMP} and $\ell_R(S/K) = \ell_R(S/R) - \ell_R(K/R)$. See \cite[Proposition 2.2]{GNO} and Proposition \ref{2.4} (1) for the second and the third equalities.

(2), (3)   The invariant $\ell_R(S/R)=\rme_1(I)$ is independent of the choice of $I$, since the first Hilbert coefficient $\rme_1(I)$ of canonical ideals $I$ is independent of the choice of $I$ (\cite[Corollary 2.8]{GMP}). Therefore, because $\ell_R(I/Q) = \rme_0(I) - \ell_R(R/I)$ depends only on $I$,  to see that $\ell_R(S/K) = \ell_R(S/R) - \ell_R(I/Q)$ is independent of the choice of $I$ and $Q$, it is enough to show that $\ell_R(K/R) = \ell_R(I/Q)$ is independent of the choice of $I$. Let $J$ be another canonical ideal of $R$ and assume that $(b)$ is a reduction of $J$.  Then, since $I \cong J$ as an $R$-module, $J = \varepsilon I$ for some invertible element $\varepsilon$ of $\rmQ(R)$ (\cite[Satz 2.8]{HK}). Let $b' = \varepsilon a$. Then $(b')$ is a reduction of $J$, $\rmr_{(b')}(J) = \rmr_Q(I)$, and $\ell_R(J/(b'))= \ell_R(I/Q)$, clearly.  Hence $\ell_R(J/(b))=\ell_R(J/(b'))= \ell_R(I/Q)$, which is independent of $I$. Because $\rmr_{(b)}(J) = \rmr_{(b')}(J)$ by Lemma \ref{2.2} (2),  the reduction number $\rmr_Q(I)$ is independent of the choice of canonical ideals $I$ and reductions $Q$ of $I$. Because $R[\frac{I}{a}]= R[\frac{J}{b'}]=R[\frac{J}{b}]$ where the second equality follows from Lemma \ref{2.2} (2), the ring $S=R[K]$ is  independent of the choice of $I$ and $Q$ as well. 
\end{proof}


\section{$2$-$\AGL$ rings and Proof of Theorem $\ref{1.4}$}\label{section3}
The main purpose of this section is to prove Theorem \ref{1.4}. Let us maintain Setting \ref{2.1}. We  begin with the following.

\begin{lemma}\label{2AGL}
The ring $R$ is $2$-$\AGL$ if and only if $K^2= K^3$ and $\ell_R(K^2/K)=2$.
\end{lemma}

\begin{proof}
If $R$ is a $2$-$\AGL$ ring, then $\ell_R(S/K) = 2$ by Theorem \ref{2.3} (1), while  by Proposition \ref{2.4} (5) $\ell_R(K^2/K) \ge 2$ since $R$ is not an $\AGL$ ring; therefore $S = K^2$. Conversely, if $K^2 = K^3$, then $K^2 = K^n$ for all $n \ge 2$, so that $S = K^2$. Hence the equivalence follows. 
\end{proof}

Before going ahead, let us note basic examples of $2$-$\AGL$ rings. Later we will give more examples. 
Let $\rmr(R) = \ell_R(\Ext_R^1(R/\m,R))$ denote the Cohen-Macaulay type of $R$.

\begin{ex}\label{ex3} Let $k[[t]]$ and $k[[X,Y,Z,W]]$ denote the formal power series rings over a field $k$. 
\begin{enumerate}[{\rm (1)}] 
\item Consider the rings $R_1 = k[[t^3,t^7,t^8]]$, $R_2 = k[[X,Y,Z,W]]/(X^3-YZ, Y^2-XZ, Z^2-X^2Y, W^2-XW)$, and $R_3= k[[t^3,t^7,t^8]] \ltimes k[[t]]$ (the idealization of $k[[t]]$ over $k[[t^3,t^7,t^8]]$). Then these rings $R_1$, $R_2$, and $R_3$ are $2$-$\AGL$ rings. The ring $R_1$ is an integral domain, $R_2$ is a reduced ring but not an integral domain, and $R_3$ is not a reduced ring.
\item Let $c \ge 4$ be an integer such that $c \not\equiv 0~\mod~3$ and set $R = k [[t^3, t^{c+3}, t^{2c}]]$. Then $R$ is a $2$-$\AGL$ ring such that $\rmv(R) = \rme (R)=3$ and $\rmr(R)=2$. 
\end{enumerate}
\end{ex}

We note basic properties of $2$-$\AGL$ rings.

\begin{prop}\label{2.7}
Suppose that $R$ is a $2$-$\AGL$ ring and set $r =\rmr(R)$. Then we have the following.
\begin{enumerate}[{\rm (1)}]
\item $\fkc = K:S =R:K$.
\item  $\ell_R(R/\fkc) = 2$. Hence there is a minimal system $x_1, x_2, \ldots, x_n$ of generators of $\fkm$ such that $\fkc = (x_1^2)+(x_2, x_3, \ldots, x_n)$.
\item $S/K \cong R/\fkc$ and $S/R \cong K/R \oplus R/\fkc$ as $R/\fkc$-modules. 
\item $K/R \cong (R/\fkc)^{\oplus \ell} \oplus (R/\m)^{\oplus m}$ as an $R/\fkc$-module for some $\ell >0$ and $m \ge 0$ such that $\ell + m = r-1$. Hence $\ell_R(K/R) = 2\ell + m$. In particular, $K/R$ is  a free $R/\fkc$-module if and only if $\ell_R(K/R) = 2(r-1)$.
\item $\mu_R(S)=r$.
\end{enumerate}
\end{prop}

\begin{proof}
(1), (2)  We have $\fkc= K:S$ and $\ell_R(R/\fkc)=2$ by Proposition \ref{2.4} (1). Hence $R:K=\fkc$, because $R :K=(K:K):K= K :K^2$. The second assertion in Assertion (2) is clear, because $\m^2 \subseteq \fkc$ and $\ell_R(\m/\fkc)=1$. 

(3), (4)  Because $R/\fkc$ is an Artinian  Gorenstein local ring, any finitely generated $R/\fkc$-module $M$ contains $R/\fkc$ as a direct summand, once $M$ is faithful. If $M$ is not faithful, then $(0):_{R/\fkc}M \supseteq \m/\fkc$ as $\ell_R(R/\fkc)=2$, so that $M$ is a vector space over $R/\fkm$. Therefore, every finitely generated $R/\fkc$-module $M$ has a  unique direct sum decomposition
$$M \cong (R/\fkc)^{\oplus \ell} \oplus (R/\fkm)^{\oplus m}$$ with $\ell, m \ge 0$ such that $\mu_{R/\fkc}(M) = \ell + m$.  Because by Assertion (1) the modules $S/R$, $K/R$, and $S/K$ are faithful  over $R/\fkc$, they contain $R/\fkc$ as a direct summand; hence $S/K \cong R/\fkc$, because $\ell_R(S/K) = \ell_R(R/\fkc)$ by Proposition \ref{2.4} (1). Consequently, the canonical exact sequence
$$0 \to K/R \to S/R \to S/K \to 0$$
of $R/\fkc$-modules is split, so that  $S/R \cong K/R \oplus R/\fkc$. Since $\mu_R(K/R) = r-1 >0$, $K/R$ contains $R/\fkc$ as a direct summand, whence
$$K/R \cong (R/\fkc)^{\oplus \ell} \oplus (R/\fkm)^{\oplus m}$$
with $\ell > 0$ and $m \ge 0$, where $\ell + m = r-1$ and $\ell_R(K/R) = 2\ell + m$.

(5) This is now clear, since   $S/R \cong K/R \oplus R/\fkc$. 
\end{proof}

The $2$-$\AGL$ rings $R$ such that $K/R$ are $R/\fkc$-free enjoy a certain specific property, which we will show in Section \ref{section4}. Here let us note Example \ref{ex5} (resp. Example \ref{ex2}) of $2$-$\AGL$ rings, for which $K/R$ is a free $R/\fkc$-module (resp. not a free $R/\fkc$-module).  Let $V=k[[t]]$ denote the formal power series ring over a field $k$.

\begin{Example}\label{ex5}
Let $e \ge 3$ and $n \ge 2$ be integers. Let $R = k[[t^e, \{t^{en+i}\}_{1 \le i \le e-2}, t^{2en-(e+1)}]]$ and $\m$ the maximal ideal of $R$. Let $K =R+\sum_{1 \le i \le e-2}Rt^{(n-2)e + i}$. Then we have the following.
\begin{enumerate}[{\rm (1)}]
\item $I=t^{2(n-1)e}K$ is a canonical ideal of $R$ containing $(t^{2(n-1)e})$ as a reduction. 
\item $R$ is a $2$-$\AGL$ ring such that $\m^2 = t^e\m$ and $\rmr(R)=e-1$.
\item $K/R \cong (R/\fkc)^{\oplus 2(e-2)}$ as an $R/\fkc$-module.
\end{enumerate}
\end{Example}

\begin{Example}\label{ex2}
Let $e \ge 4$ be an integer. Let $R=k[[t^e, \{t^{e+i}\}_{3 \le i \le e-1},t^{2e+1}, t^{2e+2}]]$ and $\m$ the maximal ideal of $R$. Let $K = R + Rt + \sum_{3 \le i \le e-1}Rt^i$. Then we have the following.
\begin{enumerate}[{\rm (1)}]
\item $I=t^{e+3}K$ is a canonical ideal of $R$ containing $(t^{e+3})$ as a reduction.
\item $R$ is a $2$-$\AGL$ ring such that $\m^2 = t^e\m$ and $\rmr(R) = e-1$.
\item $K/R \cong (R/\fkc) \oplus (R/\m)^{\oplus (e-3)}$ as an $R/\fkc$-module. 
\item $\m : \m = k[[t^3,t^4, t^5]]$.
\end{enumerate}
\end{Example}

We note the following.

\begin{thm}\label{2.6} Suppose that $R$ is a $2$-$\AGL$ ring. Then the following assertions hold true.
\begin{enumerate}[{\rm (1)}]
\item $R$ is not an $\AGL$ ring.
\item There is an exact sequence
$$0 \to \calB(-1) \to \calS_Q(I) \to \calB(-1)\to 0$$
of graded $\calT$-modules.
\item  $\fkm {\cdot} \calS_Q(I) \ne (0)$. Therefore, the above exact sequence  is not split. 
\end{enumerate}
\end{thm}

\begin{proof}
(1) Since $\ell_R(K^2/K)=2$ by Lemma \ref{2AGL}, by Proposition \ref{2.4} (5) $R$ is not an AGL ring.

(2) We have $\fkm I^2 \not\subseteq  QI$ by Proposition \ref{2.4} (4), since $I^2/QI \cong K^2/K$. Therefore, as $\ell_R(I^2/QI) = 2$, we get $\ell_R(I^2/[\m I^2+QI]) = \ell_R([\m I^2 + QI]/QI)=1$. Let us write $I^2 = QI + (f)$ for some $f \in I^2$. Then since $\ell_R([\m I^2 + QI]/QI)=1$ and $\m I^2 +QI = QI + \m f$,  $\m I^2 +QI  = QI +(\alpha f)$ for some $\alpha \in \m$. We set $g = \alpha f$. Now remember that $\calS_Q(I) = \calT{\cdot}[\calS_Q(I)]_1 =\calT{\cdot}\overline{ft}$, because $I^3 = QI^2$ (\cite[Lemma 2.1 (5)]{GNO}), where $\overline{*}$ denotes the image in $\calS_Q(I)$. Therefore, because $(0):_{\calT}\overline{gt}  = \m \calT$, we get an exact sequence
$$0 \to \calB (-1)\xrightarrow{\varphi} \calS_Q(I) \to C \to 0$$
of graded $\calT$-modules, where $\varphi (1) = \overline{gt}$.  Let $\xi$ denote the image of $\overline{ft}$ in $C = \calS_Q(I)/\calT{\cdot}\overline{gt}$. Then $C = \calT \xi$ and $(0):_{\calT}C = \m \calT$, whence $C \cong \calB(-1)$ as a graded $\calT$-module, and the result follows.

(3) Since $[\calS_Q(I)]_1 \cong I^2/QI$ as an $R$-module, we get $\m{\cdot}\calS_Q(I) \ne (0)$.
\end{proof}

We are in a position to prove Theorem \ref{1.4}.

\begin{proof}[{Proof of Theorem $\ref{1.4}$}]
(1) $\Rightarrow$ (2)  See Theorem \ref{2.6}.

(1) $\Leftrightarrow$ (3) See Lemma \ref{2AGL}.

(3) $\Leftrightarrow$ (4) Remember that $K^{n+1}/K^n \cong I^{n+1}/QI^n$ for all $n \ge 0$.

(2) $\Rightarrow$ (1) We have $\rank~\calS_Q(I) = \ell_{\calT_\fkp} ([\calS_Q(I)] _\fkp) =2{\cdot}\ell_{\calT_\fkp} (B_\fkp) = 2$, where $\fkp = \m \calT$.
 
(1) $\Leftrightarrow$ (6) $\Leftrightarrow$ (7) See Theorem \ref{2.3} (1).

(1) $\Rightarrow$ (5) By Theorem \ref {2.6} (1) $R$ is not an $\AGL$ ring. Hence $K:\m \subsetneq K^2=S$ by Corollary \ref{2.5} (1). Because $\ell_R((K:\m)/K)=1$ and $$\ell_R(S/(K:\m))+\ell_R\left((K:\m)/K\right)=\ell_R(S/K) = 2,$$ we get $\ell_R(S/(K:\m))= 1$.

See Theorem \ref{2.6} (3) for the former part of the last assertion. To see the latter part, notice that
\begin{eqnarray*}
\ell_R(R/I^{n+1})&=&\ell_R(R/Q^{n+1}) - \left[\ell_R(I^{n+1}/Q^nI) + \ell_R(Q^nI/Q^{n+1})\right]\\
&=&\rme_0(I)\binom{n+1}{1} -\left[\ell_R([\calS_Q(I)]_n) +\ell_R(I/Q)\right]\\
&=&\rme_0(I)\binom{n+1}{1} - \left[2 + \ell_R(I/Q)\right]\\
\end{eqnarray*}
for all $n \ge 1$, where the last equality follows from the exact sequence given by Condition (2). Thus 
$$\ell_R(R/I^{n+1})=\rme_0(I)\binom{n+1}{1} - \left[\rme_0(I) -\ell_R(R/I) + 2\right]$$for all $n \ge 1$.
\end{proof}

Let us note a consequence of Theorem \ref{1.4}.

\begin{proposition}\label{3.6}
Suppose that $\rmr(R)=2$. Then the following conditions are equivalent.
\begin{enumerate}[{\rm (1)}]
\item $R$ is a $2$-$\AGL$ ring.
\item $\fkc = R : K$ and $\ell_R(K/R) = 2$.
\item $S=K^2$ and $\ell_R(K/R) = 2$. 
\end{enumerate}
When this is the case, $K/R \cong R/\fkc$ as an $R$-module. 
\end{proposition}

\begin{proof}
(1) $\Rightarrow$ (2) By Proposition \ref{2.7} (4)  $K/R \cong R/\fkc$, since $\mu_R(K/R)=\rmr(R)-1=1$. Hence $\ell_R(K/R) = \ell_R(R/\fkc)=2$.

(2) $\Leftrightarrow$ (3) Remember that $S = K^2$ if and only if $\fkc = R:K$; see Proposition \ref{2.4} (2).

(2) $\Rightarrow$ (1)  Since $\mu_R(K/R) = 1$, $K/R \cong R/\fkc$, so that $\ell_R(R/\fkc) = \ell_R(K/R)=2$. Hence $R$ is a $2$-$\AGL$ ring by Theorem \ref{1.4}. 
\end{proof}

Let us explore the question of how the $2$-$\AGL$ property is preserved under flat base changes. Let $(R_1, \m_1)$ be a Cohen-Macaulay local ring of dimension one and let $\varphi : R \to R_1$ be a flat local homomorphism of local rings such that $R_1/\m R_1$ is a Gorenstein ring. Hence $\dim R_1/\m R_1 = 0$ and $\rmK_{R_1} \cong  R_1 \otimes_RK$ as an $R_1$-module (\cite[Satz 6.14]{HK}). Notice that $$R_1 \subseteq R_1 \otimes_RK \subseteq R_1 \otimes_R \overline{R} \subseteq \overline{R_1}$$
in $\rmQ(R_1)$. We set $K_1 = R_1 \otimes_RK$. Then $R_1$ also satisfies the conditions stated in Setting \ref{2.1} and we have following.

\begin{proposition}\label{5.1}
For each $n \ge 0$ the following assertions hold true.
\begin{enumerate}[{\rm (1)}]
\item $K_1^n = K_1^{n+1}$ if and only if $K^n = K^{n+1}$.
\item $\ell_{R_1}(K_1^{n+1}/K_1^n) = \ell_{R_1}(R_1/\m R_1){\cdot}\ell_{R}(K^{n+1}/K^n)$.
\end{enumerate}
\end{proposition}

\begin{proof}
The equalities  follow from the isomorphisms $$K_1^{n} \cong R_1 \otimes_R K^n,\ \ K_1^{n+1}/K_1^n \cong R_1 \otimes_R(K^{n+1}/K^n)$$ of $R_1$-modules.
\end{proof}

We furthermore have the following.

\begin{thm}\label{5.2}
The following conditions are equivalent.
\begin{enumerate}[{\rm (1)}]
\item $R_1$ is a $2$-$\AGL$ ring.
\item
Either $(\mathrm{i})$ $R$ is an $\AGL$ ring and $\ell_{R_1}(R_1/\m R_1) = 2$ or $(\mathrm{ii})$ $R$ is a $2$-$\AGL$ ring and $\m R_1 = \m_1$.
\end{enumerate}
\end{thm}

\begin{proof}
Suppose that $R_1$ is a $2$-$\AGL$ ring. Then $K_1^2 = K_1^3$ and $\ell_{R_1}(K_1^2/K_1) = 2$. Therefore, $K^2 = K^3$ and $\ell_{R_1}(K_1^{2}/K_1) = \ell_{R_1}(R_1/\m R_1){\cdot}\ell_{R}(K^{2}/K)=2$ by Proposition \ref{5.1}. We have $\ell_R(K^2/K)=1$ (resp. $\ell_R(K^2/K)= 2$) if $\ell_{R_1}(R_1/\m R_1)=2$ (resp. $\ell_{R_1}(R_1/\m R_1) = 1$), whence the implication (1) $\Rightarrow$ (2) follows. The reverse implication is now clear.
\end{proof}

\begin{ex}
Let $n \ge 1$ be an integer and let $R_1 = R[X]/(X^n + \alpha_1X^{n-1} + \cdots + \alpha_n)$, where $R[X]$ denotes the polynomial ring and $\alpha_i \in \m$ for all $1 \le i \le n$. Then $R_1$ is a flat local $R$-algebra with $\m_1 = \m R_1 + (x)$ (here $x$ denotes the image of $X$ in $R_1$) and $R_1/\m R_1 = (R/\m)[X]/(X^n)$ is a Gorenstein ring. Since $\ell_{R_1}(R_1 /\m R_1) = n$, taking $n=1$ (resp. $n=2$), we get $R_1$ is an $\AGL$ ring (resp. $R_1$ is a $2$-$\AGL$ ring). Notice that if $R$ is an integral domain and $0 \ne \alpha \in \m$, then $R_1=R[X]/(X^2 - \alpha X)$ is a reduced ring but not an integral domain. The ring $R_2$ of Example \ref{ex3} (1) is obtained in this manner,  taking $n=2$ and  $\alpha = t^3$, from the $\AGL$ ring $R = k[[t^3, t^4, t^5]]$.
\end{ex}

We say that $R$ has minimal multiplicity, if $\rmv(R) = \rme(R)$. When $\m$ contains a reduction $(\alpha)$, this condition is equivalent to saying that $\m^2 = \alpha \m$ (\cite{L,S2}).

\begin{prop}\label{3.7}
Suppose that $\rme(R)=3$ and $R$ has minimal multiplicity. Then $R$ is a $2$-$\AGL$ ring if and only if $\ell_R(K/R)=2$.
\end{prop}

\begin{proof} Thanks to Theorem \ref{5.2}, passing to $R_1= R[X]_{\m R[X]}$ if necessary, we can assume that the residue class field $R/\m$ of $R$ is infinite. Since $\rmv(R) = \rme(R)=3$, $\rmr(R)=2$. Therefore, by Corollary  \ref{3.6} we have only to show that $S = K^2$, once $\ell_R(K/R) = 2$.  Choose a non-zerodivisor $b$ of $R$ so that $J= bK \subsetneq R$. Then, since $R/\m$ is infinite, $J$ contains an element $c$ such that $J^3 = cJ^2$ (see \cite{S1, ES}; remember that $\mu_R(J^3) \le \rme(R) = 3$). Hence $K^2 = K^3$ by Lemma \ref{2.2} (1) and Theorem \ref{2.3} (2).   
\end{proof}

We close this section with the following.

\begin{remark}\label{3.9}
Let $\rmr(R) = 2$ and assume that $R$ is a homomorphic image of a regular local ring $T$ of dimension $3$. If $R$ is a $2$-$\AGL$ ring, then   $R$ has a minimal $T$-free resolution of the form
$$0 \to T^{\oplus 2} \xrightarrow{\left[\begin{smallmatrix}
X^2&g_1\\
Y&g_2\\
Z&g_3\\
\end{smallmatrix}\right]} T^{\oplus 3} \to T \to R \to 0,$$
where $X,Y,Z$ is a regular system of parameters of $T$. In fact, let $$0 \to T^{\oplus 2} \xrightarrow{\Bbb M} T^{\oplus 3} \to T \to R \to 0$$
be a minimal $T$-free resolution of $R$. Then, since $K \cong \Ext_T^2(R,T)$, taking the $T$-dual of the resolution, we get a minimal $T$-free resolution 
$$ 0 \to T \to T^{\oplus 3} \xrightarrow{{}^t\Bbb M} T^{\oplus 2} \xrightarrow{\tau} K \to 0$$
of $K$. Because $\mu_R(K/R)=1$, without loss of generality, we may assume that $\tau (\mathbf{e}_2) = 1$, where $\mathbf{e}_2 =\left(\begin{smallmatrix}
0\\
1\\
\end{smallmatrix}\right)
$. Therefore, writing ${}^t \Bbb M =\left[\begin{smallmatrix}
f_1&f_2&f_3\\
g_1&g_2&g_3\\
\end{smallmatrix}\right]
$, we get
$$K/R \cong T/(f_1, f_2,f_3)$$ as a $T$-module. Let $C = K/R$ and $\fkq = (f_1, f_2, f_3)$. Then since $\ell_T(T/\fkq) = \ell_R(C) = 2$, after suitable elementary column-transformations in the matrix ${}^t \Bbb M$ we get that $f_1 =X^2, f_2=Y, f_3 = Z$ for some regular system $X, Y, Z$ of parameters of $T$.

The converse of the assertion in Remark \ref{3.9} is not true in general. In the case where $R$ is the semigroup ring of a numerical semigroup, we shall give in Section \ref{section5} a complete description of the assertion in terms of the matrix ${}^t\Bbb M = \left[\begin{smallmatrix}
f_1&f_2&f_3\\
g_1&g_2&g_3\\
\end{smallmatrix}\right]$ (see Theorem \ref{7.4} and its consequences).
\end{remark}


\section{$2$-$\AGL$ rings obtained by idealization}
In this section we study the problem of when the idealization $A = R \ltimes \fkc$ of $\fkc=R:S$ over $R$ is a $2$-$\AGL$ ring. To do this we need some preliminaries. For a moment let $R$ be an arbitrary commutative ring and $M$ an $R$-module. Let $A=R \ltimes M$ be the idealization of $M$ over $R$. Hence $A = R \oplus M$ as an $R$-module and the multiplication in $A$ is given  by
$$(a,x)(b,y) = (ab, bx + ay)$$
where $a,b \in R$ and $x,y \in M$.
Let $K$ be an $R$-module and set $L = \Hom_R(M,K) \oplus K$. We consider $L$ to be an $A$-module under the following action of $A$
$$(a,x)\circ (f,y) = (af, f(x) + ay),$$
where $(a,x) \in A$ and $(f,y) \in L$. Then it is standard to check that the map
$$\Hom_R(A,K) \to L, ~\alpha \mapsto (\alpha \circ j, \alpha (1))$$
is an isomorphism of $A$-modules, where $j : M \to A,~ x \mapsto (0,x)$ and $1 = (1,0)$ denotes the identity of the ring $A$.

We are now back to our Setting \ref{2.1}. Let $A = R \ltimes \fkc$ and set $L = S \times K$. Then $A$ is a one-dimensional Cohen-Macaulay local ring and $$\rmK_A = \Hom_R(A,K) \cong \Hom_R(\fkc,K) \times K \cong L=S \times K,$$
because $\fkc = K:S$ by Proposition \ref{2.4} (1). Therefore 
$$  A =R \ltimes \fkc \subseteq L = S \times K \subseteq \overline{R} \ltimes \rmQ(R).$$ 
Because $\rmQ(A) = \rmQ(R) \ltimes \rmQ(R)$ and $\overline{A}= \overline{R} \ltimes \rmQ(R)$, our idealization $A=R \ltimes \fkc$ satisfies the same assumption as in Setting \ref{2.1} and we have the following.

\begin{proposition}\label{4.1}
The following assertions hold true.
\begin{enumerate}[{\rm (1)}]
\item $L^n = S \ltimes S$ for all $n \ge 2$, whence $A[L] = S \ltimes S$.
\item $\ell_A(A[L]/L) =\ell_R(S/K)$.
\item $A : A[L] = \fkc \times \fkc$.
\item $\rmv(A) = \rmv(R) + \mu_R(\fkc)$ and $\rme(A) = 2 {\cdot}\rme(R)$.
\end{enumerate}
\end{proposition}

\begin{proof}
(1) Since $L^n = (S \times K)^n = S^n \times S^{n-1}K$, we have  $L^n = S \times S$  for $n \ge 2$.

(2) We get $\ell_A(A[L]/L) = \ell_R((S \oplus S)/(S \oplus K)) = \ell_R(S/K)$.

(3) This is straightforward, since $A[L] = S \ltimes S$.

(4) To see the first assertion, remember that $\m \times \fkc$ is the maximal ideal of $A$ and that
$(\m \times \fkc)^2 = \m^2 \times \m \fkc.$ For the second one, notice that $\m A$ is a reduction of $\m \times \fkc$ and that $A = R \oplus \fkc$ as an $R$-module. We then have $\rme(A) = \rme_{\m}^0(A) = 2{\cdot}\rme^0_\m(R)$.
\end{proof}

By Proposition \ref{4.1} (2) we readily get the following.

\begin{thm}\label{4.2}
$A = R \ltimes \fkc$ is a $2$-$\AGL$ ring if and only if so is $R$.
\end{thm}

\begin{ex}\label{ex4}
Suppose that $R$ is a $2$-$\AGL$ ring, for instance, take $R = k[[t^3, t^7, t^8]]$ (see Example \ref{ex3} (1)). We set $$
A_n =
\begin{cases} 
R & ~if~n=0,\\
A_{n-1} \ltimes \fkc_{n-1} & ~if~n \ge 1,\\
\end{cases}
$$
that is $A_0=R$ and for $n \ge 1$ let $A_n$ be the idealization of $\fkc_{n-1}$ over $A_{n-1}$, where $\fkc_{n-1} = A_{n-1} : A_{n-1}[K_{n-1}]$ and $K_{n-1}$ denotes a fractional ideal of $A_{n-1}$ such that $A_{n-1} \subseteq K_{n-1} \subseteq \overline{A_{n-1}}$ and $K_{n-1} \cong \rmK_{A_{n-1}}$ as an $A_{n-1}$-module. We then have an infinite family $\{A_n\}_{n \ge 0}$ of analytically ramified $2$-$\AGL$ rings such that $\rme(A_n) = 2^n{\cdot}\rme(R)$ for each $n \ge 0$.  Since $\fkc = t^6k[[t]] \cong k[[t]]$  for $R = k[[t^3,t^7,t^8]]$, the ring $R_3= k[[t^3,t^7,t^8]] \ltimes k[[t]]$ of Example \ref{ex3} (1) is obtained in this manner.
\end{ex}


\section{The algebra $\fkm:\fkm$}\label{section4}
We maintain Setting \ref{2.1} and set $B = \m :\m$. By \cite[Theorem 5.1]{GMP}  $B$ is a Gorenstein ring if and only if $R$ is an $\AGL$ ring of minimal multiplicity. Our purpose of this section is to explore the structure of the algebra $B = \m : \m$ in connection with the $2$-$\AGL$ property of $R$.

Let us begin with the following.

\begin{prop}\label{finalprop}
Suppose that there is an element $\alpha \in \m$ such that $\m^2 = \alpha \m$ and that $R$ is not a Gorenstein ring. Set $L = BK$. Then the following assertions hold true.
\begin{enumerate}[$(1)$]
\item $B = R: \m = \frac{\m}{\alpha}$ and $K:B = \m K$.
\item $L = K :\m$, $L \cong \m K$ as a $B$-module, and $B \subseteq L \subseteq \overline{B}$.
\item  $S = B[L]=B[K]$.
\end{enumerate}
\end{prop}

\begin{proof} Since $R$ is not a DVR (resp. $\m^2 = \alpha\m$), we have $B = R: \m$ (resp. $B = \frac{\m}{\alpha}$). Because $K :\m K = R : \m = B$, we get $\m K = K : B$. We have $K: L = R : B = \m$, since $R \subsetneq B$. Therefore, $L = K : \m$. Clearly, $L=\frac{\m K}{\alpha} \cong \m K$ as a $B$-module. We have $S \subseteq B[K] \subseteq B[L]$. Because $B \subseteq L=K:\m \subseteq K^2$ by Corollary \ref{2.5} (1), $B[L] \subseteq S$, whence $S = B[L] = B[K]$ as claimed.
\end{proof}

We have the following.

\begin{thm}\label{thm6.2}
Suppose that $R$ is a $2$-$\AGL$ ring. Assume that there is an element $\alpha \in \m$ such that $\m^2 = \alpha \m$. Set $L = BK$. Then the following assertions hold true.
\begin{enumerate}[{\rm (1)}]
\item $\ell_R(L^2/L)=1$. 
\item Let $M = (0):_B(L^2/L)$. Then $M \in \Max B$, $R/\m \cong B/M$, and $B_M$ is an $\mathrm{AGL}$ ring which is not a Gorenstein ring. 
\item If $N \in \Max B \setminus \{M\}$, then $B_N$ is a Gorenstein ring.\end{enumerate}
Therefore, $B_N$ is an $\mathrm{AGL}$ ring for every $N \in \Max B$.
\end{thm}

\begin{proof} Because $S = K^2$ and $S \supseteq L \supseteq K$ by Proposition \ref{finalprop}, we have $S = L^2$, while $\ell_R(L^2/L)=1$ as  $L = K :\m$. Hence $$0 < \ell_{B/M}(L^2/L) = \ell_B(L^2/L) \le \ell_R(L^2/L) =1,$$ so that $M \in \Max B$, $R/\m \cong B/M$, and $L^2/L \cong B/M$. Because $L\cong K:B$ as a $B$-module by Proposition \ref{finalprop}, we get $L_M \cong \rmK_{B_M}$ as a $B_M$-module (\cite[Satz 5.12]{HK}). Therefore, since $\ell_{B_M}(L_M^2/L_M) = 1$, by \cite[Theorem 3.16]{GMP} $B_M$ is an AGL ring which is not a Gorenstein ring. If $N \in \Max B$ and if $N \ne M$, then $(L^2/L)_N = (0)$, so that $B_N$ is a Gorenstein ring by \cite[Theorem 3.7]{GMP}.
\end{proof}

Let us note a few consequences of Theorem \ref{thm6.2}.

\begin{corollary}\label{6.2a} Assume that $\m^2 = \alpha \m$ for some $\alpha \in \m$ and that $B$ is a local ring with maximal ideal $\n$. Then the following conditions are equivalent.
\begin{enumerate}[{\rm (1)}]
\item $R$ is a $2$-$\AGL$ ring.
\item $B$ is a non-Gorenstein $\AGL$ ring  and $R/\m \cong B/\n$.
\end{enumerate}
When this is the case, $S$ is a Gorenstein ring, provided $\rmv(B) = \rme(B)$.
\end{corollary}

\begin{proof}
By Theorem \ref{thm6.2} we have only to show the implication (2) $\Rightarrow$ (1). Let $L = BK$. Then $L = K:\m$, $L \cong \rmK_B$, and $S = B[L]$ by Proposition \ref{finalprop}. Because $B$ is a non-Gorenstein $\AGL$ ring, $\ell_B(B[L]/L)= 1$ by \cite[Theorem 3.16]{GMP}, so that $$\ell_R(S/K)=\ell_R(S/L) + \ell_R(L/K) = \ell_B(B[L]/L) + \ell_R((K : \m)/K)= 2,$$
where the second equality follows from the fact that $R/\m \cong B/\n$. Hence $R$ is a $2$-$\AGL$ ring. The last assertion is a direct consequence of \cite[Theorem 5.1]{GMP}.
\end{proof}

If $R$ is the semigroup ring of a numerical semigroup, the algebra $B = \m : \m$ is also the semigroup ring of a numerical semigroup, so that $B$ is always a local ring with $R/\m \cong  B/\n$, where $\n$ denotes the maximal ideal of $B$. Hence by Corollary \ref{6.2a} we readily get the following. Let $k[[t]]$ be the formal power series ring over a field $k$.

\begin{corollary}
Let $H=\left<a_1, a_2, \ldots, a_\ell\right>$ be a numerical semigroup and $R=k[[t^{a_1}, t^{a_2}, \ldots, t^{a_\ell}]]$ the semigroup ring of $H$. Assume that $R$ has minimal multiplicity. Then $R$ is a $2$-$\AGL$ ring if and only if $B = \m :\m$ is an $\AGL$ ring which is not a Gorenstein ring. When this is the case, $S$ is a Gorenstein ring, provided $\rmv(B) = \rme(B)$.
\end{corollary}

\begin{proof}
Remember that $\m^2 = t^e\m$, where $e = \min\{a_i \mid 1 \le i \le \ell \}$.
\end{proof}

If $\rmv(R) < \rme(R)$, the ring $S$ is not necessarily a Gorenstein ring, even though $R$ is a $2$-$\AGL$ ring and $B$ is an $\AGL$ ring with $\rmv(B) = \rme(B)\ge 3$. Let us note one example.

\begin{Example}\label{ex1}
Let $R = k[[t^5,t^7,t^9,t^{13}]]$ and set $K =R+Rt^3$. Then we have the following.
\begin{enumerate}[{\rm (1)}]
\item $K \cong \rmK_R$ as an $R$-module and $I=t^{12}K$ is a canonical ideal of $R$ with $(t^{12})$ a reduction. Hence $\rmr(R)=2$. 
\item $S = k[[t^3, t^5, t^7]]$ and $\fkc = (t^{10})+(t^7, t^9, t^{13})$. \item[$(3)$] $R$ is a $2$-$\AGL$ ring with $\rmv(R) = 4$ and $\rme(R) = 5$. 
\item $K/R \cong R/\fkc$ as an $R$-module. 
\item $B = k[[t^5, t^7, t^8,t^9,t^{11}]]$ and $B$ is an $\AGL$ ring,  possessing minimal multiplicity $5$.
\end{enumerate}
\end{Example}

The ring $B$ does not necessarily have minimal multiplicity, even though $B$ is a local ring and $R$ is a $2$-$\AGL$ ring of minimal multiplicity. Let us note one example.

\begin{Example}\label{ex1}
Let $R = k[[t^4,t^9,t^{11},t^{14}]]$ and set $K =R+Rt^3+Rt^5$. Then we have the following.
\begin{enumerate}[{\rm (1)}]
\item $K \cong \rmK_R$ as an $R$-module and $I=t^{11}K$ is a canonical ideal of $R$ with $(t^{11})$ a reduction. Hence $\rmr(R)=3$. 
\item $R$ is a $2$-$\AGL$ ring with $\m^2 = t^4\m$.
\item $\ell_R(K/R)=3$ and $K/R \cong R/\fkc \oplus R/\m$ as an $R$-module.
\item $B = k[[t^4,t^5,t^7]]$.
\end{enumerate}
\end{Example}

In Theorem \ref{thm6.2}, if $K/R$ is a free $R/\fkc$-module, then $B = \m:\m$ is necessarily a local ring. To state the result, we need further notation.

Suppose that $R$ is a $2$-$\AGL$ ring and set $r = \rmr(R)$. Then since by Proposition \ref{2.7} (4) $$K/R \cong (R/\fkc)^{\oplus \ell}\oplus (R/\m)^{\oplus m}$$ with integers $\ell >0$ and $m \ge 0$ such that $\ell + m = r-1$, there are elements $f_1, f_2, \ldots, f_{\ell}$ and $g_1, g_2, \ldots, g_m$ of $K$ such that
$$K/R= \sum_{i = 1}^\ell R{\cdot}\overline{f_i} \bigoplus \sum_{j = 1}^m R{\cdot}\overline{g_j}, \ \ \sum_{i = 1}^\ell R{\cdot}\overline{f_i} \cong (R/\fkc)^{\oplus \ell}, \ \ and \ \ \sum_{j = 1}^m R{\cdot}\overline{g_j} \cong (R/\m)^{\oplus m}$$ where $\overline{*}$ denotes the image in $K/R$.   We set $F =\sum_{i = 1}^\ell Rf_i $ and $U = \sum_{j = 1}^m Rg_j$. Let us write $\fkc = (x_1^2) + (x_2, x_3, \ldots, x_n)$ for some minimal system $\{x_i\}_{1 \le i \le n}$ of generators of $\m$ (see Proposition \ref{2.7} (2)). With this notation we have the following.

\begin{proposition}\label{6.3} The following assertions hold true.
\begin{enumerate}[{\rm (1)}]
\item Suppose $m=0$, that is $K/R$ is a free $R/\fkc$-module. Then $B$ is a local ring with maximal ideal $\m S$ and  $R/\m \cong B/\m S$.
\item Suppose that $U^q \subseteq \m S$ for some $q > 0$. Then $B$ is a local ring. 
\end{enumerate}
\end{proposition}

\begin{proof} We divide the proof of Proposition \ref{6.3} into a few steps. Notice that $B = R:\m$, since $R$ is not a DVR.

\begin{claim}\label{claim622} The following assertions hold true.
\begin{enumerate}[{\rm (1)}]
\item $\m^2 S \subseteq R$.
\item $\m S \subseteq \rmJ(B)$, where $\rmJ(B)$ denotes the Jacobson radical of $B$.
\end{enumerate}
\end{claim}

\begin{proof}
Because $\m^2 K^3 = \m^2 K^2 \subseteq K$, we have $\m^2 K^2 \subseteq K :K = R$. Hence $\m^2 S \subseteq R$, so that $\m S \subseteq R : \m = B$. Let $M \in \Max B$ and choose $N \in \Max S$ such that $M = N \cap B$. Then because $\m S \subseteq N$, $\m S \subseteq N \cap B=M$, whence $\m S \subseteq \rmJ(B)$. 
\end{proof}

We consider Assertion (2) of Proposition \ref{6.3}. Since $g_j^q \in \m S$ for all $1 \le j \le m$, the ring $B/\m S =(R/\m)[\overline{g_1},\overline{g_2}, \ldots, \overline{g_m}]$ is a local ring, where $\overline{g_j}$ denotes the image of $g_j$ in $B/\m S$. Therefore $B$ is a local ring, since $\m S \subseteq \rmJ(B)$ by Claim \ref{claim622} (2).

To prove Assertion (1) of Proposition \ref{6.3}, we need more results.  Suppose that $m=0$; hence $U=(0)$ and $\ell = r-1$.

\begin{claim}\label{claim621} The following assertions hold true.
\begin{enumerate}[{\rm (1)}]
\item $x_1 f_i \not\in R$ for all $1 \le i \le \ell$.
\item $(R:\m) \cap K = R +x_1F$ and $\ell_R([(R:\m)\cap K]/R) = \ell$.
\item $B = R + x_1F +Rh$ for some $h \in \m K^2$.
\end{enumerate}
\end{claim}

\begin{proof}
(1) Remember that $K/R=\sum_{i = 1}^\ell R{\cdot}\overline{f_i} \cong (R/\fkc)^{\oplus \ell}$. We then have $\m \overline{f_i} \ne (0)$ but $\fkc \overline{f_i} =(0)$ for each $1 \le i \le \ell$. Hence $x_1 f_i \not\in R$, because $\fkc = (x_1^2)+ (x_j \mid 2 \le j \le n)$ and $\m = (x_1)+\fkc$.

(2) Because $(0):_{R/\fkc}\m$ is generated by the image of $x_1$, we have $$(0):_{K/R}\m = \sum_{i = 1}^\ell R{\cdot}\overline{x_1f_i},$$ whence  $(R:\m) \cap K = R +x_1F$ and $\ell_R([(R:\m)\cap K]/R) =\ell$.

(3) Notice that by Claim \ref{claim622} (1)
$$R \subseteq \m K + R \subseteq (R:\m) \cap K \subseteq R:\m$$ and that $\ell_R([(R:\m)\cap K]/R)= \ell = r-1$ by Assertion (2).  We then have
$$\ell_R((R:\m)/[(R:\m) \cap K]) = 1,$$ because $\ell_R((R:\m)/R) = r$. Hence $R:\m = [(R:\m) \cap K] + Rg$ for some $g \in R:\m$. On the other hand, because
$K \subsetneq \m K^2 + K \subsetneq K^2$ by Corollary \ref{2.5} (1) and $\ell_R(K^2/K)=2$, we get $\ell_R(K^2/(\m K^2 + K))=\ell_R((\m K^2 + K)/K)=1$. Consequently, $K^2 = K + R\xi$ for some $\xi \in K^2$. Let us write $g = \rho + \beta \xi$ with $\rho \in K$ and $\beta \in \m$. Then since $\beta \xi \in \m K^2 \subseteq R:\m$ by Claim \ref{claim622} (1) and $g \in R:\m$, we get $\rho \in (R:\m) \cap K$, whence setting $h = \beta \xi$, we have $$B = R:\m = [(R:\m)\cap K] + Rg = [(R:\m)\cap K] + Rh$$
as claimed.  
\end{proof}

Let us finish the proof of Assertion (1) of Proposition \ref{6.3}. In fact, by Claim \ref{claim621} (3) we have $B \subseteq  R + \m S$, which implies $R/\m \cong B/\m S$, whence $\m S \in \Max B$. Therefore, $B$ is a local ring with unique maximal ideal $\m S$, since $\m S \subseteq \rmJ(B)$ by Claim \ref{claim622} (2). This completes the proof of Proposition \ref{6.3}.
\end{proof}

Under some additional conditions, the converse of Theorem \ref{thm6.2} is also true.

\begin{thm}\label{final}
Suppose that the following conditions are satisfied.
\begin{enumerate}[$(1)$]
\item $\rme(R) = e \ge 3$ and $R$ is not an $\mathrm{AGL}$ ring,
\item $B$ is an $\mathrm{AGL}$ ring with $\rme(B) = e$, and
\item there is an element $\alpha \in \m$ such that $\m^2 = \alpha \m$ and $\n^2 = \alpha \n$, where $\n$ denotes the maximal ideal of $B$. 
\end{enumerate}
Then $R$ is a $2$-$\AGL$ ring and $K/R$ is a free $R/\c$-module. 
\end{thm}

\begin{proof} We have $B = R :\m = \frac{\m}{a}$. Let $L = B K$. Then by Proposition \ref{finalprop} (3) $S = B[L]=B[K]$. As $B$ is an AGL ring, we have $S = \n : \n$ by \cite[Theorem 3.16]{GMP}, whence $S = \frac{\n}{\alpha}$. Consequently, $R \subseteq B= \frac{\m}{\alpha} \subseteq S= \frac{\n}{a}$. Let us write $\m = (\alpha, x_2, x_3, \ldots, x_e)$. We set $y_i = \frac{x_i}{\alpha}$ for each $2 \le i \le e$.

\begin{claim}
We can choose the elements $\{x_i\}_{2 \le i \le e}$ of $\m$ so that $y_i \in \n$ for all $2 \le i \le e$.
\end{claim}

\begin{proof}
Since by Conditions (1) and (2) $$\ell_B(B/\alpha B) = \rme(B) = e = \rme(R) = \ell_R(R/\alpha R) = \ell_R(B/\alpha B),$$ we get the isomorphism $R/\m \cong B/\n$ of fields. Let $2 \le i \le e$ be an integer and choose $c_i \in R$ so that $y_i \equiv c_i ~\mod ~\n$. Then since $y_i -c_i = \frac{x_i -\alpha c_i}{\alpha} \in \n$, replacing $x_i$ with $x_i -\alpha c_i$, we have $y_i \in \n$ for all $2 \le i \le e$.
\end{proof}

We now notice that $B =  \frac{\m}{\alpha} = R + \sum_{i = 2}^eR{\cdot}\frac{x_i}{\alpha}$ and $\frac{x_i}{\alpha}
 \in \fkn$ for each $2 \le i \le e$. We then have $$\fkn = (\fkn \cap R) + \sum_{i = 2}^eR{\cdot}\frac{x_i}{\alpha} = \fkm +\sum_{i = 2}^eR{\cdot}\frac{x_i}{\alpha} = R\alpha + \sum_{i = 2}^eR{\cdot}\frac{x_i}{\alpha},$$ where the last equality follows from the fact that $\m = R\alpha + \sum_{i = 2}^eRx_i$. Thus $$S = \frac{\n}{\alpha} = R + \sum_{i =2}^eR{\cdot}\frac{x_i}{\alpha^2},$$  whence $\alpha^2 \in \fkc$. Let $2 \le i,j \le e$ be integers. Then $x_i{\cdot}\frac{x_j}{\alpha^2} = \frac{x_i}{\alpha}{\cdot}\frac{x_j}{\alpha} \in \n^2 = \alpha \n$ and $\n = R\alpha + \sum_{i = 2}^eR{\cdot}\frac{x_i}{\alpha}$, so that  $x_i{\cdot}\frac{x_j}{\alpha^2} \in R\alpha^2 + \sum_{i=2}^eRx_i$, which shows $(\alpha^2, x_2, x_3, \ldots, x_e) \subseteq \fkc$. Therefore,  $\fkc = (\alpha^2, x_2, x_3, \ldots, x_e)$ because  $\fkc \subsetneq \m$ (remember that $R$ is not an AGL ring), whence  $\ell_R(R/\fkc) = 2$, so that  $R$ is a $2$-$\AGL$ ring. Because $S =\frac{\n}{\alpha}$ and $B = \frac{\m}{\alpha}$ and because $R/\m \cong B/\n$, we get $$\ell_R(S/B) =\ell_R(S/\n) - \ell_R(B/\n) = \ell_R(S/\alpha S) - 1 = e-1 ~~\text{and}$$
 $$\ell_R(B/R) = \ell_R(B/\fkm) - \ell_R(R/\m) = \ell_R(B/\alpha B) - 1 = e-1.$$ Therefore,  $\ell_R(S/R) = 2e-2$, whence $\ell_R(K/R) = 2e -4 = 2(e-2)$ because  $\ell_R(S/K) = 2$. Consequently,  by Proposition \ref{2.7} (4) $K/R$ is a free $R/\fkc$-module, since $\mu_R(K/R) = e-2$ (notice that $\rmr(R) = e-1$), which completes the proof of Theorem \ref{final}.
\end{proof}

However, the ring $B$ is not necessarily a local ring in general, although $R$ is a $2$-$\AGL$ ring with $\rmv(R) = \rme(R)$. Let us note one example.

\begin{ex}\label{ex}
Let $V = k[[t]]$ be the formal power series ring over an infinite field $k$. We consider the direct product $A = k[[t^3, t^7, t^8]] \times  k[[t^3, t^4, t^5]]$ of rings and set $R= k{\cdot}(1,1) + \rmJ(A)$ where $\rmJ(A)$ denotes the Jacobson radical of $A$. Then $R$ is a subring of $A$ and a one-dimensional Cohen-Macaulay complete local ring with $\rmJ(A)$ the maximal ideal. We have the ring $R$ is a $2$-$\AGL$ ring and $\rmv(R) = \rme(R) = 6$. However $$\fkm : \fkm = k[[t^3,t^4, t^5]] \times V$$ which is not a local ring, so that $K/R$ is not a free $R/\fkc$-module.
\end{ex}

\section{Numerical semigroup rings}\label{section5}

Let $k$ be a field.  In this section we study the case where $R = k[[H]]$ is the semigroup ring of a numerical semigroup $H$.  First of all, let us fix the notation, according to the terminology of numerical semigroups.

\begin{setting}\label{7.1}
Let $0 < a_1, a_2, \ldots, a_\ell \in \Bbb Z~(\ell >0)$ be positive integers such that $\mathrm{GCD}~(a_1, a_2, \ldots, a_\ell)=1$. We set $$H = \left<a_1, a_2, \ldots, a_\ell\right>=\left\{\sum_{i=1}^\ell c_ia_i \mid 0 \le c_i \in \Bbb Z~\text{for~all}~1 \le i \le \ell \right\}$$
and call it the numerical semigroup generated by the numbers $\{a_i\}_{1 \le i \le \ell}$. Let $V = k[[t]]$ be the formal power series ring over a field $k$. We set
$$R = k[[H]] = k[[t^{a_1}, t^{a_2}, \ldots, t^{a_\ell}]]$$
in $V$ and call it the semigroup ring of $H$ over $k$. The ring  $R$ is a one-dimensional Cohen-Macaulay local domain with $\overline{R} = V$ and $\m = (t^{a_1},t^{a_2}, \ldots, t^{a_\ell} )$.

We set $T = k[t^{a_1}, t^{a_2}, \ldots, t^{a_\ell}]$ in $k[t]$. Let $P = k[X_1, X_2, \ldots, X_\ell]$ be the polynomial ring over $k$. We consider  $P$ to be a $\Bbb Z$-graded ring such that $P_0 = k$ and $\deg X_i = a_i$ for  $1 \le i \le \ell$. Let $\varphi : P \to T$ denote the homomorphism of graded $k$-algebras defined by $\varphi (X_i) = t^{a_i}$ for each $1 \le i \le \ell$.  
\end{setting}

In this section we are interested in the question of when $R=k[[H]]$ is a $2$-$\AGL$ ring. To study the question, we recall some basic notion on numerical semigroups. Let $$\rmc(H) = \min \{n \in \Bbb Z \mid m \in H~\text{for~all}~m \in \Bbb Z~\text{such~that~}m \ge n\}$$ be the conductor of $H$ and set $\rmf(H) = \rmc(H) -1$. Hence $\rmf(H) = \max ~(\Bbb Z \setminus H)$, which is called the Frobenius number of $H$. Let $$\mathrm{PF}(H) = \{n \in \Bbb Z \setminus H \mid n + a_i \in H~\text{for~all}~1 \le  i \le \ell\}$$ denote the set of pseudo-Frobenius numbers of $H$. Therefore, $\rmf(H)$ equals the $\rma$-invariant of the graded $k$-algebra $k[t^{a_1}, t^{a_2}, \ldots, t^{a_\ell}]$ and $\sharp \mathrm{PF}(H) = \rmr(R)$  (\cite[Example (2.1.9), Definition (3.1.4)]{GW}).  We set  $f = \rmf(H)$ and $$K = \sum_{c \in \mathrm{PF}(H)}Rt^{f-c}$$ in $V$. Then $K$ is a fractional ideal of $R$ such that $R \subseteq K \subseteq \overline{R}$ and $$K \cong \rmK_R = \sum_{c \in \mathrm{PF}(H)}Rt^{-c}$$ as an $R$-module (\cite[Example (2.1.9)]{GW}). Let us refer to $K$ as the fractional canonical ideal of $R$. Notice that  $t^f \not\in K$ but $\m t^f \subseteq R$, whence $K:\m = K + Rt^f$.

Let us begin with the following.

\begin{thm}\label{7.2}
Suppose that $\ell \ge 3$ and $a_j \not\in \left<a_1, \ldots, \overset{\vee}{a_j}, \ldots, a_\ell \right>$ for all $1 \le j \le \ell$.  Assume that $\rmr(R)=2$ and let $K = R+Rt^a$ for some $0 < a \in \Bbb Z$.  Then the following conditions are equivalent.
\begin{enumerate}[{\rm (1)}]
\item $R$ is a $2$-$\AGL$ ring.
\item $3a \in H$ and  $f = 2a + a_i$ for some $1 \le i \le \ell$.
\end{enumerate}
\end{thm}

\begin{proof}
(1) $\Rightarrow$ (2)
We have $t^{3a} \in K^2$ as $K^2 = K^3$. If $2a \in H$, then $K^2=K$ and $R$ is a Gorenstein ring (Proposition \ref{2.4} (3)). Hence $2a \not\in H$, so that $3a \in H$. Because $K : \m = \m K^2 + K$ by Corollary \ref{2.5} (2), we get $K + Rt^f = K:\m=K + \sum_{j=1}^\ell Rt^{2a + a_j}$. Therefore, because $\ell_R((K:\m)/K)= 1$, 
$$K + Rt^f = K + Rt^{2a+a_i}$$
for some $1 \le i \le \ell$, whence $f = 2a + a_i$.

(2) $\Rightarrow$ (1)
We get $K^3 = K^2=K + Rt^{2a}$, since $3a \in H$. Let $L = \m K^2 + K$. Hence $L = K + \sum_{j=1}^\ell Rt^{2a+a_j}$ and $L \subsetneq K^2$ because  $R$ is not a Gorenstein ring. Notice that $\ell_R(K^2/L) = 1$, since $\mu_R(K^2/K)=1$. We furthermore have the following.

\begin{claim}\label{claim7.2.1}
$L = K:\m$.
\end{claim}

\begin{proof}[Proof of Claim $\ref{claim7.2.1}$]
We have only to show $L \subseteq K:\m = K + Rt^{2a+a_i}$. Let $1 \le j \le \ell$ and assume that $t^{2a+a_j} \not\in K + Rt^{2a+a_i}$. Then $a_j < a_i$ since $f = 2a + a_i = \rmc(H) -1$. Because $2a + a_j \not\in H$, $t^{f -(2a + a_j)} = t^{a_i -a_j} \in K= R + Rt^a$. Hence $a_i - a_j \in H$ or $a_i - a_j -a \in H$. Suppose that $a_i - a_j -a \in H$. Then, setting  $h = a_i - a_j -a \in H$, we get $a_i = a_j +a + h$ whence $f = 2a + a_i = 3a + a_j + h \in H$, which is impossible. Hence  $a_i - a_j \in H$. Let us write
$$a_i - a_j = \sum_{k=1}^\ell m_k a_k$$
with $0 \le m_k \in \Bbb Z$. Then $m_k = 0$ if $a_k \ge a_i$, since $a_i - a_j < a_i$. Therefore
$$a_i = a_j + \sum_{a_k < a_i}m_ka_k \in \left<a_1, \ldots, \overset{\vee}{a_i}, \ldots, a_\ell\right>$$
which contradicts the assumption that $a_j \not\in \left<a_1, \ldots, \overset{\vee}{a_j}, \ldots, a_\ell \right>$ for all $1 \le j \le \ell$. Thus $L = K + Rt^{2a+a_i}$.
\end{proof}
We have  $\ell_R(K^2/K)= \ell_R(K^2/L) + \ell_R(L/K) = 2$ because $\ell_R(L/K)= 1$ by Claim \ref{claim7.2.1}, so that  $R$ is a $2$-$\AGL$ ring.
\end{proof}

Let us recover Example \ref{ex3} (2) in the present context.

\begin{corollary}\label{7.3}
Let $c \ge 4$ such that $c \not\equiv~0~\mod~3$ and set  $H = \left<3, c+3, 2c\right>$. Then $R$ is a $2$-$\AGL$ ring such that $\rmr(R) = 2$ and $K/R \cong R/\fkc$ as an $R$-module.
\end{corollary}

\begin{proof}
We set $a = c-3$. Then $f = 2a + 3$ and $K = R + Rt^{a}$.
\end{proof}

Suppose that $\ell = 3$. We set $\fka = \Ker \varphi$, where $$\varphi : P =k[X_1, X_2, X_3] \to T= k[t^{a_1}, t^{a_2}, t^{a_3}]$$ is the homomorphism of $k$-algebras defined by $\varphi(X_i) = t^{a_i}$ for $i=1,2,3$. Let us write  $X = X_1$, $Y=X_2$, and $Z = X_3$ for short. If $T$ is not a Gorenstein ring, then  by \cite{H} it is known that $\fka = \rmI_2\left(\begin{smallmatrix}
X^\alpha&Y^\beta&Z^\gamma\\
Y^{\beta'}&Z^{\gamma'}&X^{\alpha'}\\
\end{smallmatrix}
\right)$ for some integers $\alpha, \beta, \gamma, \alpha', \beta', \gamma' > 0$, where $ \rmI_2\left(\begin{smallmatrix}
X^\alpha&Y^\beta&Z^\gamma\\
Y^{\beta'}&Z^{\gamma'}&X^{\alpha'}\\
\end{smallmatrix}
\right)$ denotes the ideal of $P$ generated by the $2 \times 2$ minors of the matrix $\left(\begin{smallmatrix}
X^\alpha&Y^\beta&Z^\gamma\\
Y^{\beta'}&Z^{\gamma'}&X^{\alpha'}\\
\end{smallmatrix}
\right)$.  With this notation we have the following.

\begin{thm}\label{7.4} Suppose that $H$ is $3$-generated, that is $\ell = 3$. Then the following conditions are equivalent.
\begin{enumerate}[{\rm (1)}]
\item $R$ is a $2$-$\AGL$ ring.
\item After a suitable permutation of $a_1, a_2, a_3$, $\fka = \rmI_2\left(\begin{smallmatrix}
X^2&Y&Z\\
Y^{\beta'}&Z^{\gamma'}&X^{\alpha'}\\
\end{smallmatrix}
\right)$ for some integers $\alpha', \beta', \gamma'$ such that $\alpha' \ge 2$ and $\beta', \gamma' >0$.
\end{enumerate}
\end{thm}

To prove Theorem \ref{7.4}, we need a result of \cite[Section 4]{GMP}. Throughout, let $H = \left<a_1, a_2, a_3\right>$ and assume that $T$ is not a Gorenstein ring. Hence the ideal $\fka$ is generated by the $2 \times 2$ minors of the matrix  
$$\begin{pmatrix}   
X^{\alpha} & Y^{\beta} & Z^{\gamma} \\
Y^{\beta'} & Z^{\gamma'} & X^{\alpha'}
\end{pmatrix}$$
where $0< \alpha, \beta, \gamma, \alpha', \beta', \gamma'\in \mathbb Z$. Let $\Delta_1 = Z^{\gamma+\gamma'}-X^{\alpha'}Y^{\beta}$, $\Delta_2=X^{\alpha+\alpha'}-Y^{\beta'}Z^{\gamma}$, and $\Delta_3=Y^{\beta+\beta'}-X^{\alpha}Z^{\gamma'}$. Then $\fka  = (\Delta_1,\Delta_2, \Delta_3)$ and thanks to the theorem of Hilbert--Burch (\cite[Theorem 20.15]{E}), the graded ring $T=P/\fka$ possesses a graded minimal $P$-free resolution of the  form
$$0\longrightarrow \begin{matrix} P(-m )\\ \oplus\\ P(-n)\end{matrix} \overset{\left[ \begin{smallmatrix}
X^{\alpha} & Y^{\beta'}\\
Y^{\beta} & Z^{\gamma'}\\
Z^{\gamma} & X^{\alpha'}
\end{smallmatrix} \right]}{\longrightarrow  } \begin{matrix} P(-d_1)\\ \oplus\\ P(-d_2)\\ \oplus\\ P(-d_3)\end{matrix} \overset{\left[\Delta_1~\Delta_2~\Delta_3\right]}{\longrightarrow } P\overset{\varepsilon }{\longrightarrow } T \longrightarrow 0,$$
where $d_1 = \deg\Delta_1 = a_3(\gamma + \gamma')$, $d_2 = \deg\Delta_2 = a_1(\alpha + \alpha')$, $d_3 = \deg\Delta_3 = a_2(\beta + \beta')$, $m = a_1\alpha + d_1 = a_2\beta + d_2 = a_3\gamma + d_3$, and $n = a_1\alpha' + d_3 = a_2\beta' + d_1 = a_3\gamma' + d_2$. Therefore $$(E)\ \ \ n - m   = a_2\beta' -a_1\alpha = a_3\gamma' - a_2\beta = a_1\alpha' -a_3\gamma.$$ Let $\rmK_P = P(-d)$ denote the graded canonical module of $P$ where $d=a_1 + a_2 + a_3$. Then, taking the $\rmK_P$--dual of the above resolution, we get the  minimal presentation
\begin{equation*}
\begin{matrix} P(d_1-d)\\ \oplus\\ P(d_2-d)\\ \oplus\\ P(d_3-d)\end{matrix}\overset{\left[ \begin{smallmatrix} X^{\alpha} & Y^{\beta} & Z^{\gamma} \\
Y^{\beta'} & Z^{\gamma'} & X^{\alpha'}
\end{smallmatrix} \right]}{\longrightarrow }\begin{matrix} P(m  -d)\\ \oplus\\ P(n-d)\end{matrix}\overset{\varepsilon }{\longrightarrow }\rmK_{T}\longrightarrow 0 \tag{$\sharp$}
\end{equation*}
of the graded canonical module $\rmK_{T} = \Ext_P^2(T,\rmK_P)$ of $T$. Therefore, because $\rmK_T = \sum_{c \in \mathrm{PF}(H)}Tt^{-c}$ (\cite[Example (2.1.9)]{GW}), we have $\ell_k([\rmK_T]_i) \le 1$ for all $i \in \Bbb Z$, whence $m \ne n$. After the permutation of $a_2$ and $a_3$ if necessary, we may assume without loss of generality that $m < n$. Then the presentation $(\sharp)$ shows that $\mathrm{PF}(H) = \{m-d, n-d\}$ and $f = n-d$.

We set $a = n-m$. Hence  $a > 0$, $f = a + (m-d)$, and $K = R + Rt^a$. With this notation we have the following. Remember that $R$ is the $MT_M$-adic completion of the local ring $T_M$, where $M= (t^{a_i}\mid i = 1,2,3)$ denotes the graded maximal ideal of $T$.

\begin{prop}[{\cite[Theorem 4.1]{GMP}}]\label{7.5}
$\ell_R(K/R)=\alpha\beta\gamma$.
\end{prop}

\noindent
Therefore, if $R$ is a $2$-$\AGL$ ring, then $\ell_R(K/R)=2$ by Proposition \ref{3.6}, so that $\alpha =2$ and $\beta=\gamma = 1$ by Proposition \ref{7.5} after a suitable permutation of $a_1, a_2, a_3$ if necessary. Consequently, Theorem \ref{7.4} is reduced to the following.

\begin{thm}\label{7.6} 
Let  $m < n$ and assume that $\fka = \rmI_2\left(\begin{smallmatrix}
X^2&Y&Z\\
Y^{\beta'}&Z^{\gamma'}&X^{\alpha'}\\
\end{smallmatrix}
\right)$ with $\alpha', \beta', \gamma' >0$. Then $R$ is a $2$-$\AGL$ ring if and only if $\alpha' \ge 2$. When this is the case, $f = 2a +a_1$, where $a= n-m$
\end{thm}

\begin{proof}[{Proof of Theorem $\ref{7.6}$}] Notice that $R$ is not an $\AGL$ ring, since $\ell_R(K/R)= 2$ by Proposition \ref{7.5}. We get   by equations $(E)$ above
\begin{enumerate}
\item[{\rm (i)}] $a_2\beta' =2a_1 + a$, 
\item[{\rm (ii)}] $a_3\gamma' = a_2 + a$, and 
\item[{\rm (iii)}] $a_1\alpha' = a_3 + a$.
\end{enumerate} Suppose that $\alpha' \ge 2$. Then $3a = a_2(\beta'-1) + a_1(\alpha' -2) + a_3(\gamma'-1) \in H$. Hence $S =K^2= R + Rt^a + Rt^{2a}$. Therefore, because 
\begin{enumerate}
\item[{\rm (iv)}]  $2a_1 + 2a= (2a_1 +a)+a = a_2\beta' + (a_3\gamma' - a_2) = a_2(\gamma'-1) +a_3\gamma'\in H$,
\item[{\rm (v)}] $a_2 + 2a =(a_3\gamma' - a_2) +(a_1\alpha' - a_3) + a_2 = a_3(\gamma' -1) + a_1\alpha'\in H$, and 
\item[{\rm (vi)}] $a_3 +2a = (a_1\alpha' - a_3) + (a_2\beta' - 2a_1) + a_3 = a_1(\alpha'-2) + a_2\beta' \in H,$
\end{enumerate} 
we get that $(t^{2a_1})+(t^{a_2}, t^{a_3}) \subseteq K:S =\fkc$ by Proposition \ref{2.4} (1) and that 
$(2a+a_1)+a_i \in H$ for $i = 1,2,3$. Hence $2a + a_1 \in \mathrm{PF}(H)$ if $2a+a_1 \not\in H$. Now notice that $\m K^2+K = K + Rt^{2a+a_1}$ because $2a + a_i \in H$ for $i = 2,3$ by equations (v) and (vi), whence $t^{2a + a_1} \not\in K$ because $\m K^2 \not\subseteq K$ by Proposition \ref{2.4} (4). In particular, $2a + a_1 \not\in H$. Therefore, $t^{a_1} \not\in \fkc$, so that $\fkc \ne \m$ and $\fkc = (t^{2a_1})+(t^{a_2}, t^{a_3})$. Thus $R$ is a $2$-$\AGL$ ring, because  $\ell_R(R/\fkc)= 2$. Notice that $2a + a_1 \in \mathrm{PF}(H)=\{f-a,f\}$ and we get $f = 3a + a_1 \in H$ if $f \ne 2a+a_1$, which is impossible as $3a \in H$. Hence $f = 2a + a_1$.

Conversely, assume that $R$ is a $2$-$\AGL$ ring. Then $2a \not\in H$, since $K \ne K^2$. Therefore $3a \in H$, since $t^{3a} \in  K^2$. Because $\m K^2 + K = K + \sum_{j=1}^3Rt^{2a+a_j}$ and $a_2 + 2a \in H$ by equation (v), we get $$K + Rt^f = K:\m=\m K^2 + K = K + \left(Rt^{2a+a_1}+Rt^{2a+a_3}\right),$$ where the second equality follows from Corollary \ref{2.5} (2). Therefore, if $t^{2a+a_3} \not\in K$, then $f = 2a + a_3$, so that $\mathrm{PF}(H) = \{a+a_3, 2a+a_3\}$, which is absurd because $a + a_3  \in H$ by equation (iii). Thus $t^{2a+a_3} \in K$, so that $\m K^2 + K = K + Rt^{2a+a_1}$ and $f = 2a + a_1$. Suppose now that $\alpha' = 1$. Then $a_1 = a + a_3$ by equation (iii), whence $f=2a + a_1 = 3a + a_3 \in H$ because $3a \in H$. This is a required contradiction, which completes the proof of Theorem \ref{7.4} as well as that of Theorem \ref{7.6}.  
\end{proof}

When $H$ is $3$-generated and $\rme(R)=\min\{a_1, a_2, a_3\}$ is small, we have the following structure theorem  of $H$ for $R$ to be a $2$-$\AGL$ ring.

\begin{corollary}\label{7.7} Let $\ell = 3$. 
\begin{enumerate}[{\rm (1)}]
\item Suppose that $\min\{a_1, a_2, a_3\} = 3$. Then the following conditions are equivalent.
   \begin{enumerate}[{\rm (a)}]
   \item $R$ is a $2$-$\AGL$ ring.
   \item $H= \left<3, c+3, 2c \right>$ for some $c \ge 4$ such that $c \not\equiv 0~mod~3$.
   \end{enumerate}
\item If $\min\{a_1, a_2, a_3\} = 4$, then $R$ is not a $2$-$\AGL$ ring.
\item Suppose that $\min\{a_1, a_2, a_3\} = 5$. Then the following conditions are equivalent.
   \begin{enumerate}[{\rm (a)}]
   \item $R$ is a $2$-$\AGL$ ring.
   \item $\mathrm{(i)}$ $H= \left<5, 3c +8, 2c+2\right>$ for some $c \ge 2$ such that $c \not\equiv 4 ~mod~5$ or  $\mathrm{(ii)}$  $H= \left<5, c+4, 3c + 2\right>$ for some $c \ge 2$ such that $c \not\equiv 1 ~mod~5$.
   \end{enumerate}
\end{enumerate}
\end{corollary}

\begin{proof} Let $e = \min \{a_1, a_2, a_3\}$. Suppose that $R$ is a $2$-$\AGL$ ring. Then by Theorem \ref{7.4}, after a suitable permutation of $a_1, a_2, a_3$ we get $$\fka = \rmI_2\left(\begin{smallmatrix}
X^2&Y&Z\\
Y^{\beta'}&Z^{\gamma'}&X^{\alpha'}\\
\end{smallmatrix}
\right)$$ for some integers $\alpha', \beta', \gamma'$ such that $\alpha' \ge 2$ and $\beta', \gamma' >0$. Remember that 
$$a_1= \beta' \gamma' + \beta' + 1,$$ 
because $a_1 =\ell_R(R/t^{a_1}R) = \ell_k(k[Y,Z]/(Y^{\beta' + 1}, Y^{\beta'}Z, Z^{\gamma'+1})$. We similarly have that
$$a_2 = 2 \gamma'+ \alpha' \gamma' + 2 \ge 6, \ \ \ a_3 = \alpha' \beta' + \alpha' + 2\ge 6$$
since $\alpha' \ge 2$. Therefore,  $e = a_1= \beta' \gamma' + \beta' + 1$, if $e \le 5$.

(1) (a) $\Rightarrow$ (b) We have  $\beta' = \gamma' = 1$. Hence $a_2 =\alpha' + 4$ and $a_3 = 2\alpha'+2$, that is $H = \left<3, c+3, 2c\right>$, where $c = \alpha' + 1$. We have $c \not\equiv~0~\mod~3$ because $\mathrm{GCD}~(3, c +3, 2c)=1$, whence $c \ge 4$.

(b) $\Rightarrow$ (a) See Corollary \ref{7.3} or notice that $\fka = \rmI_2\left(\begin{smallmatrix}
X^2&Y&Z\\
Y &Z &X^{c-1}\\
\end{smallmatrix}
\right)$
and apply Theorem \ref{7.4}

(2) We have $a_1= \beta' \gamma' + \beta' + 1 =4$, so that 
$\beta'= 1$ and  $\gamma' = 2$. Hence $a_2 = 2\alpha' + 6$ and $a_3 = 2\alpha' + 2$, which is impossible because $\mathrm{GCD}~(a_1, a_2, a_3) = 1$.

(3) (a) $\Rightarrow$ (b) We set $c = \alpha'$. Then either $\beta' = 1$ and $\gamma'=3$ or $\beta' = 2$ and $\gamma'=1$. For the former case  we get (i) $H = \left<5, 3c + 8, 2c +2\right>$, where $c \not\equiv~4~\mod~5$. For the latter case we get (ii) $H= \left<5, c+4, 3c + 2\right>$, where $c \not\equiv 1 ~\mod~5$.

(b) $\Rightarrow$ (a) For Case (i) we have $\fka = \rmI_2\left(\begin{smallmatrix}
X^2&Y&Z\\
Y&Z^3&X^c\\
\end{smallmatrix}
\right)$ and for Case (ii) $\fka = \rmI_2\left(\begin{smallmatrix}
X^2&Y&Z\\
Y^{2}&Z&X^{c}\\
\end{smallmatrix}
\right)$, whence $R$ is a $2$-$\AGL$ ring by Theorem \ref{7.4}.
\end{proof}

Even though $R$ is a $2$-$\AGL$ ring, $K/R$ is not necessarily a free $R/\fkc$-module (Example \ref{ex2}). Here let us note a criterion for the freeness of $K/R$.

\begin{proposition}\label{7.9}
Let $r =\rmr(R) \ge 2$ and write $\mathrm{PF}(H) = \{c_1 < c_2< \cdots < c_r=f\}$.  Assume that $R$ is a $2$-$\AGL$ ring. Then the following conditions are equivalent.
\begin{enumerate}[{\rm (1)}]
\item $K/R \cong (R/\fkc)^{\oplus (r-1)}$ as an $R$-module.
\item There is an integer $1 \le j \le \ell$ such that $f + a_j = c_i + c_{r-i}$ for all $1 \le i \le r-1$.
\end{enumerate}
\end{proposition}

\begin{proof}
(2) $\Rightarrow$ (1)  We have $K = \sum_{i=1}^{r}Rt^{f-c_i}$ and $\ell_R(K/(\m K + R))= \mu_R(K/R) = r-1$. Because $t^{c_{r-i}} = t^{f-c_i + a_j}\in \m K + R$ for all $1 \le i < r$, $R + \sum_{i=1}^{r-1}Rt^{c_i} \subseteq \m K + R$, whence $$\ell_R(K/R) \ge \ell_R(K/(\m K + R)) + \ell_R((R+\sum_{i=1}^{r-1}Rt^{c_i})/R) = 2(r-1).$$ Thus $K/R$ is a free $R/\fkc$-module by Proposition \ref{2.7} (4).

(1) $\Rightarrow$ (2)
We may assume that $a_j \not\in \left<a_1, \ldots, \overset{\vee}{a_j}, \ldots, a_\ell\right>$ for all $1 \le j \le \ell$. Hence $\m$ is minimally generated by  the elements $\{t^{a_i}\}_{1 \le i \le \ell}$. Therefore, since  $\ell_R(R/\fkc) = 2$, by Proposition \ref{2.7} (2)  $\fkc = (t^{2a_j}) + (t^{a_i} \mid 1 \le i \le \ell, i \ne j)$ for some $1 \le j \le \ell$. Because $K/R$ is minimally generated by $\{\overline{t^{f-c_i}} \}_{1 \le i \le r-1}$ where $\overline{t^{f-c_i}}$ denotes the image of $\overline{t^{f-c_i}}$ in $K/R$  and  because $K/R$ is a free $R/\fkc$-module, the homomorphism
$$(\sharp)\ \ \varphi : (R/\fkc)^{\oplus (r-1)} \to K/R, \ \ \ \mathbf{e}_i \mapsto \overline{t^{f-c_i}}\ \  ~\text{for~each}~1 \le i \le r-1$$
of $R/\fkc$-modules has to be an isomorphism, where $\{\mathbf{e}_i\}_{1 \le i \le r-1}$ denotes the standard basis of $(R/\fkc)^{\oplus (r-1)}$. Now remember that $t^{a_j} \not\in \fkc$, which shows via the isomorphism $(\sharp)$ above that $t^{a_j}{\cdot}t^{f-c_i} \not\in R$ for all $1 \le i \le r-1$, while we have $t^{2a_j}{\cdot}t^{f-c_i} \in R$ and $t^{a_k}{\cdot}t^{f-c_i} \in R$ for all $k \ne j$. Therefore, $f-c_i + a_j \not\in H$ but $(f-c_i + a_j)+a_m \in H$ for all $1 \le m \le \ell$, so that $f-c_i + a_j \in \mathrm{PF}(H)$ for all $1 \le i \le r-1$. Notice that $f-c_1 + a_j \le f$ because $f - c_1 + a_j \not\in H$ and that  $f-c_1 + a_j < f$ because $c_1 \ne a_j$. Therefore,  because 
$$\{f- c_{r-1} + a_j < \cdots < f-c_2 + a_j < f-c_1 + a_j \} \subseteq \mathrm{PF}(H)=\{c_1 < \cdots < c_{r-1} < c_r =f\}$$
and  $f-c_1 + a_j  < f$, we readily get that $f+a_j = c_i + c_{r-i}$ for all $1 \le i \le r-1$.
\end{proof}

We close this paper with a broad method of constructing $2$-$\AGL$ rings from a given symmetric numerical semigroup. Let $H$ be a numerical semigroup and assume that $H$ is symmetric, that is $R=k[[H]]$ is a Gorenstein ring. For the rest of this paper we fix an arbitrary integer $0 < e \in H$. Let $\alpha_i = \min \{h \in H \mid h \equiv i ~\mod ~e\}$ for each $0 \le i \le e-1$. We set $Ap_e(H) = \{\alpha_i \mid 0 \le i \le e-1\}$. We then have $Ap_e(H) = \{h \in H \mid h - e \not\in H\}$,  which is called the Apery set of $H$ mod $e$. Let us write $$Ap_e(H) = \{h_0 = 0< h_1 < \cdots < h_{e-1}\}.$$   We then have $h_{e-1} = h_i + h_{e-1-i}$
for all $0 \le i \le e-1$, because $H$ is symmetric. Notice that  $H = \left<e, h_1, \ldots, h_{e-2}\right>$ and $h_{e-1} = h_1 + h_{e-2}$. Let $n \ge 0$ be an integer and set $$H_n =\left<e, h_1+ne, h_2+ne, \ldots, h_{e-1}+ ne \right>.$$
Notice that   $H_0 =H$ and for each $n >0$, $$e < h_1 + ne < \cdots < h_{e-2} + ne < h_{e-1}+ne$$ and $\mathrm{GCD}~(e, h_1 + ne, \ldots, h_{e-2} + ne, h_{e-1}+ne)=1$. We set
$$R_n = k[[H_n]], \ \ S_n = R_n[K_n],\ \ \text{and}\ \ \fkc_n = R_n : S_n,$$
where $K_n$ denotes the fractional canonical ideal of $R_n$. Let $\m_n = (t^e)+(t^{h_i +ne} \mid 1 \le i \le e-1)$ be the maximal ideal of $R_n$.

With this notation we have the following.

\begin{thm}\label{7.10}
For all $n \ge 0$ the following assertions hold true.
\begin{enumerate}[{\rm (1)}]
\item $K_n^2 = K_n^3$.
\item $\ell_{R_n}(K_n^2/K_n) = n$.
\item $K_n/R_n \cong (R_n/\fkc_n)^{\oplus (e-2)}$ as an $R_n$-module.
\end{enumerate}
Hence $R_2$ is a $2$-$\AGL$ ring for any choice of the integer $0 < e \in H$.
\end{thm}

\begin{proof}
We may assume $n > 0$. We begin with the following.

\begin{claim}\label{7.11} The following assertions hold true.
\begin{enumerate}[{\rm (1)}]
\item $h + ne \in H_n$ for all $h \in H$.
\item $\rmv(R_n) = \rme(R_n) = e$.
\end{enumerate}
\end{claim}

\begin{proof}[{Proof of Claim $\ref{7.10}$}]
(1) Let $h = h_i + qe$ with $0 \le i \le e-1$ and $q \ge 0$. Then $h + ne = (h_i + ne) + qe \in H_n$.

(2) Let $1 \le i,j \le e-1$. Then $(h_i + ne) + (h_j + ne) - e= \left[(h_i + h_j) + ne \right] + (n-1)e \in H_n$ by Assertion (1). Therefore, $\m_n^2 = t^e\m_n$.
\end{proof}

Consequently, by Claim \ref{7.11} (2) we get that $\{e\} \cup \{h_i + ne\}_{1 \le i \le e-1}$ is a minimal system of generators of $H_n$, whence $$\mathrm{PF}(H_n) = \{h_1 + (n-1)e, h_2 + (n-1)e, \ldots, h_{e-1}+ (n-1)e\}.$$ Therefore, $K_n = \sum_{j=0}^{e-2}R_nt^{h_j}$, so that $S_n = R_n[K_n]=R$. 

Let  $0 \le i,j \le e-1$ and write $h_i + h_j = h_k + qe$ with $0 \le k \le e-1$ and $q \ge 0$. If $k \le e-2$, then $t^{h_i}t^{h_j} = (t^e)^qt^{h_k} \in K_n$, which shows $K_n^2 = K_n + R_nt^{h_{e-1}}$ (remember that $h_{e-1}= h_1 + h_{e-2}$). Hence $K_n^3 = K_n^2 + \sum_{i=1}^{e-2}R_n{\cdot}t^{h_{e-1}+h_i}$. If $1 \le i \le e-2$ and $t^{h_{e-1}+h_i} \not\in K_n$, then $t^{h_{e-1}+h_i} \in R_n t^{h_{e-1}} \subseteq K_n^2$ as we have shown above. Hence $K_n^2 = K_n^3$, which proves Assertion (1) of Theorem \ref{7.10}.

Because $S_n = R$, we have $\m_n R = t^eR$, so that $R = \sum_{j=0}^{e-1}R_nt^{h_j}$. Now notice that by Claim \ref{7.11} (1) $$(h_i + ne) + h_j = (h_i+ h_j) + ne \in H_n$$ for all $0 \le i,j \le e-1$ and we get $t^{h_i+ne} \in \fkc_n$, whence $t^{ne}R_n +(t^{h_i + ne} \mid 1 \le i \le e-1)R_n \subseteq \fkc_n$, while $(n-1)e + h_j \not\in H_n$ for all $1 \le j \le e-1$,  so that $t^{(n-1)e} \not\in \fkc_n$. Thus $$\fkc_n  =t^{ne}R_n + (t^{h_i + ne} \mid 1 \le i \le e-1)R_n = (t^{h_i + ne} \mid 0 \le i \le e-1)R_n$$ and hence $\ell_{R_n}(R_n/\fkc_n) =  n$. Therefore,  $\ell_{R_n}(K_n/R_n)= n$ by Proposition \ref{2.4} (1), which proves Assertion (2) of Theorem \ref{7.10}.

To prove Assertion (3) of Theorem \ref{7.10}, it suffices by Assertion (2) that $\ell_{R_n}(K_n/R_n) = n(e-2)$, because $\fkc_n = R_n : K_n$ by Proposition \ref{2.4} (2) and $\mu_{R_n}(K_n/R_n) = e-2$ (notice that $\rmr(R_n) = e-1$ by Claim \ref{7.11} (2)). We set $L_q = \m_n^q  K_n + R_n$. We then have by induction on $q$ that $L_q = R_n + \sum_{j=1}^{e-2}R_nt^{h_j + qe}$ for all  $0 \le q \le n$. In fact, let $0 \le q < n$ and assume that our assertion holds true for $q$. Then since $L_{q+1} = \m_n L_q + R_n$, we get $L_{q+1} = R_n + \m_n \left[\sum_{j=1}^{e-2}R_nt^{h_j + qe}\right]$. Therefore, $L_{q+1} = R_n + \sum_{j=1}^{e-2}R_nt^{h_j + (q+1)e}$, because for all $1 \le i \le e-1$ and $1 \le j \le e-2$
$$(h_i + ne) + (h_j + qe) = \left[(h_i + h_j) + ne\right] + qe \in H_n$$
by Claim \ref{7.11} (1).  
Hence we obtain a filtration
$$K_n =L_0 \supseteq L_1 \supseteq \cdots \supseteq L_n = R_n,$$
where $L_{q}= L_{q+1} + \sum_{j=1}^{e-2}R_nt^{h_j + qe}$ and $\m_n{\cdot}(L_q/L_{q+1}) = (0)$ for $0 \le q < n$. Consequently, to see that $\ell_{R_n}(K_n/R_n) = n(e-2)$, it is enough to show the following.

\begin{claim}\label{7.12}
$\ell_{k}(L_q/L_{q+1}) = e-2$ for all $0 \le q < n$.
\end{claim}

\begin{proof}[{Proof of Claim $\ref{7.12}$}]
Let $0 \le q < n$ and let $\{c_j\}_{1 \le j \le e-2}$ be elements of the field $k$ such that $\sum_{j=1}^{e-2}c_jt^{h_j + qe}\in L_{q+1}$. Suppose $c_j \ne 0$ for some $1 \le j \le e-2$. Then $t^{h_j+qe} \in L_{q+1}$. Hence  $h_j + qe \in H_n$ or $(h_j + qe)-(h_m + (q+1)e) \in H_n$  for some $1 \le m \le e-2.$ We get $h_j + qe \not\in H_n$, since $h_j + (n-1)e \not\in H_n$. On the other hand, if $$(h_j + ne)-(h_m + ne + e) =  (h_j + qe) - (h_m + (q+1)e) \in H_n,$$ then $1 \le m < j \le e-2$. Let us  write
$$(h_j + ne) - (h_m + ne + e)= \alpha_0 e + \alpha_1 (h_1 + ne) + \cdots + \alpha_{e-1} (h_{e-1}+ ne)$$
with integers $0 \le \alpha_p \in \Bbb Z$. Then  $\alpha_j = 0$ since $ (h_j + ne) - (h_m + ne + e) < h_j + ne$, so that
$$h_j + ne =  (\alpha_0 + 1)e + \alpha_1 (h_1 + ne) + \cdots + \overset{\vee}{\alpha_j (h_j + ne)} +  \cdots + \cdots + \alpha_{e-1}(h_{e-1} + ne),$$  
which violates the fact that $\{e\} \cup \{h_i + ne\}_{1 \le i \le e-1}$ is a minimal system of generators of $H_n$. Thus $c_j = 0$ for all $1 \le j \le e-2$, whence $\ell_{k}(L_q/L_{q+1}) = e-2$ as claimed.
\end{proof}

Therefore, $\ell_{R_n}(K_n/R_n) = n(e-2)$, so that $K_n/R_n \cong (R_n/\fkc_n)^{\oplus (e-2)}$ as an $R_n$-module, which completes the proof of Theorem \ref{7.10}.
\end{proof}



\end{document}